\documentclass[a4paper,11pt]{article}
\usepackage{aligned-overset}
\usepackage{amsmath,amsthm}
\usepackage{authblk}
\usepackage[style=numeric-comp,maxbibnames=99,bibencoding=utf8,giveninits=true,backend=biber]{biblatex}
\usepackage{booktabs}
\usepackage{cancel}
\usepackage{cases}
\usepackage{colortbl}
\usepackage[hmargin={26mm,26mm},vmargin={30mm,35mm}]{geometry}
\usepackage{nicefrac}
\usepackage[colorlinks,allcolors={blue}]{hyperref}
\usepackage{paralist}
\usepackage{subcaption}
\usepackage{enumitem}
\usepackage[compact,small]{titlesec}
\usepackage{tikz-cd}
\usetikzlibrary{arrows}
\usepackage{xcolor}

\usepackage{newtxtext}
\usepackage{newtxmath}

\bibliography{ddr-compactness}

\newcommand{\email}[1]{\href{mailto:#1}{#1}}

\numberwithin{equation}{section}

\newtheorem{theorem}{Theorem}

\newtheorem{lemma}[theorem]{Lemma}

\theoremstyle{remark}
\newtheorem{remark}[theorem]{Remark}
\theoremstyle{definition}
\newtheorem{assumption}[theorem]{Assumption}

\newcommand{\Real}{\mathbb{R}}

\DeclareRobustCommand{\bvec}[1]{\boldsymbol{#1}}
\newcommand{\uvec}[1]{\underline{\bvec{#1}}}

\DeclareMathOperator{\GRAD}{\mathbf{grad}}
\DeclareMathOperator{\CURL}{\mathbf{curl}}
\DeclareMathOperator{\DIV}{div}

\newcommand{\gensymbol}{\bullet}
\newcommand{\Hgrad}[1]{H(\GRAD,#1)}
\newcommand{\Hcurl}[1]{\boldsymbol{H}(\CURL,#1)}
\newcommand{\Hdiv}[1]{\boldsymbol{H}(\DIV,#1)}
\newcommand{\Hgen}[1]{H(\gensymbol,#1)}
\newcommand{\HgradG}[1]{H_{\Gamma}(\GRAD,#1)}
\newcommand{\HgradS}[1]{H_{\Sigma}(\GRAD,#1)}
\newcommand{\HcurlG}[1]{\boldsymbol{H}_{\Gamma}(\CURL,#1)}
\newcommand{\HcurlS}[1]{\boldsymbol{H}_{\Sigma}(\CURL,#1)}
\newcommand{\HdivG}[1]{\boldsymbol{H}_{\Gamma}(\DIV,#1)}
\newcommand{\HdivS}[1]{\boldsymbol{H}_{\Sigma}(\DIV,#1)}
\newcommand{\HdivGc}[1]{\boldsymbol{H}_{\Gamma^\compl}(\DIV,#1)}
\newcommand{\HgenG}[1]{H_{\Gamma}(\gensymbol,#1)}
\newcommand{\HgenS}[1]{H_{\Sigma}(\gensymbol,#1)}

\newcommand{\HcurlGcmu}[1]{\boldsymbol{H}_{\Gamma^\compl}(\CURL_\mu,#1)}
\newcommand{\Hcurlmu}[1]{\boldsymbol{H}(\CURL_\mu,#1)}

\newcommand{\HdivGcmu}[1]{\boldsymbol{H}_{\Gamma^\compl}(\DIV_\mu,#1)}
\newcommand{\Hdivmu}[1]{\boldsymbol{H}(\DIV_\mu,#1)}

\newcommand{\Hgradmu}[1]{H(\GRAD_\mu,#1)}

\newcommand{\HgenSmu}[1]{H_{\Sigma}(\gensymbol_\mu,#1)}
\newcommand{\Hgenemptymu}[1]{H_{\emptyset}(\gensymbol_\mu,#1)}
\newcommand{\Hgenmu}[1]{H(\gensymbol_\mu,#1)}

\newcommand{\compl}{{\rm c}}

\newcommand{\symbolproj}{\pi}
\newcommand{\lproj}[2]{\symbolproj_{\mathcal{P},#2}^{#1}}
\newcommand{\genlproj}[2]{\symbolproj_{\gensymbol,#2}^{#1}}
\newcommand{\RTproj}[2]{\symbolproj_{\RTsymbol,#2}^{#1}}
\newcommand{\NEproj}[2]{\symbolproj_{\NEsymbol,#2}^{#1}}

\newcommand{\Xgrad}[2]{\underline{X}_{\GRAD}^{#1}(#2)}
\newcommand{\Xcurl}[2]{\underline{X}_{\CURL}^{#1}(#2)}
\newcommand{\Xdiv}[2]{\underline{X}_{\DIV}^{#1}(#2)}
\newcommand{\Xgen}[2]{\underline{X}_{\gensymbol}^{#1}(#2)}   
\newcommand{\XgradG}[2]{\underline{X}_{\GRAD,\Gamma}^{#1}(#2)}
\newcommand{\XcurlG}[2]{\underline{X}_{\CURL,\Gamma}^{#1}(#2)}
\newcommand{\XdivG}[2]{\underline{X}_{\DIV,\Gamma}^{#1}(#2)}

\newcommand{\dof}[2]{\underline{#1}_{#2}}
\newcommand{\vdof}[2]{\underline{\bvec{#1}}_{#2}}

\newcommand{\interpsymbol}{\underline{I}}
\newcommand{\Igrad}[2]{\interpsymbol_{\GRAD,#2}^{#1}}
\newcommand{\Icurl}[2]{\interpsymbol_{\CURL,#2}^{#1}}
\newcommand{\Idiv}[2]{\interpsymbol_{\DIV,#2}^{#1}}
\newcommand{\Igen}[2]{\interpsymbol_{\gensymbol,#2}^{#1}}

\newcommand{\Lagsymbol}{\GRAD}
\newcommand{\NEsymbol}{\CURL}
\newcommand{\RTsymbol}{\DIV}
\newcommand{\Lag}[2]{V_{\Lagsymbol}^{#1}(#2)}
\newcommand{\NE}[2]{\boldsymbol{V}_{\NEsymbol}^{#1}(#2)}
\newcommand{\RT}[2]{\boldsymbol{V}_{\RTsymbol}^{#1}(#2)}
\newcommand{\FEgen}[2]{{V}_{\gensymbol}^{#1}(#2)}
\newcommand{\LagG}[2]{V_{\Lagsymbol,\Gamma}^{#1}(#2)}
\newcommand{\NEG}[2]{\boldsymbol{V}_{\NEsymbol,\Gamma}^{#1}(#2)}
\newcommand{\RTG}[2]{\boldsymbol{V}_{\RTsymbol,\Gamma}^{#1}(#2)}
\newcommand{\FEgenG}[2]{{V}_{\gensymbol,\Gamma}^{#1}(#2)}

\newcommand{\DFE}{\mathcal D_{\rm FE}}
\newcommand{\DDDR}{\mathcal D_{\rm DDR}}

\newcommand{\QFEinterpsymbol}{\widehat{\mathcal J}}
\newcommand{\qILag}[2]{\QFEinterpsymbol_{\Lagsymbol,#2}^{#1}}
\newcommand{\qINE}[2]{\QFEinterpsymbol_{\NEsymbol,#2}^{#1}}
\newcommand{\qIRT}[2]{\QFEinterpsymbol_{\RTsymbol,#2}^{#1}}
\newcommand{\qIFEgen}[2]{\QFEinterpsymbol_{\gensymbol,#2}^{#1}}

\newcommand{\Qinterpsymbol}{\underline{\widehat{I}}}
\newcommand{\qIgrad}[2]{\Qinterpsymbol_{\GRAD,#2}^{#1}}
\newcommand{\qIcurl}[2]{\Qinterpsymbol_{\CURL,#2}^{#1}}
\newcommand{\qIdiv}[2]{\Qinterpsymbol_{\DIV,#2}^{#1}}
\newcommand{\qIgen}[2]{\Qinterpsymbol_{\gensymbol,#2}^{#1}}

\newcommand{\uG}[2]{\uvec{G}_{#2}^{#1}}
\newcommand{\uC}[2]{\uvec{C}_{#2}^{#1}}
\newcommand{\D}[2]{D_{#2}^{#1}}

\newcommand{\term}{\mathfrak T}
\newcommand{\Rcurl}{\bvec{R}_{\CURL}}
\newcommand{\Rgrad}{R_{\GRAD}}
\newcommand{\Liftcurl}{\bvec{\mathcal L}_{\CURL,h}}
\newcommand{\tr}[2]{\gamma_{#2}^{#1}}
\newcommand{\trt}[2]{\bvec{\gamma}_{\mathrm{t},#2}^{#1}}

\newcommand{\Pgrad}[2]{P_{\GRAD,#2}^{#1}}
\newcommand{\Pcurl}[2]{\bvec{P}_{\CURL,#2}^{#1}}
\newcommand{\Pdiv}[2]{\bvec{P}_{\DIV,#2}^{#1}}
\newcommand{\Pgen}[2]{P_{\gensymbol,#2}^{#1_\gensymbol}}

\newcommand{\normal}{\bvec{n}}
\newcommand{\tangent}{\bvec{t}}

\newcommand{\letterPoly}{\mathcal{P}}
\newcommand{\Poly}[1]{\letterPoly^{#1}}
\newcommand{\vPoly}[1]{\boldsymbol{\letterPoly}^{#1}}

\newcommand{\norm}[2]{\|#2\|_{#1}}
\newcommand{\seminorm}[2]{|#2|_{#1}}
\newcommand{\vvvert}{\vert\kern-0.25ex\vert\kern-0.25ex\vert}
\newcommand{\tnorm}[2]{\vvvert #2\vvvert_{#1}}

\newcommand{\Appr}[2]{\mathsf{A}^{#1}_{#2}}
\newcommand{\Bppr}[2]{\mathsf{B}^{#1}_{#2}}
\newcommand{\Tppr}[2]{\mathsf{T}^{#1}_{#2,\Gamma^\compl}}

\DeclareMathOperator{\Image}{Im}

\newcommand{\Mh}{\mathfrak{M}_h}
\newcommand{\Th}{\mathcal{T}_h}
\newcommand{\Fh}{\mathcal{F}_h}
\newcommand{\Eh}{\mathcal{E}_h}
\newcommand{\Vh}{\mathcal{V}_h}
\newcommand{\sfP}{\mathsf{P}}
\newcommand{\Fhb}{\mathcal{F}_h^{\partial\Omega}}
\newcommand{\FGc}{\mathcal{F}_h^{\Gamma^\compl}}

\newcommand{\MSh}{\mathfrak{S}_h}
\newcommand{\Sh}{\mathcal{S}_h}
\newcommand{\patch}[1]{\widehat{#1}}
\newcommand{\ppatch}[1]{\widetilde{#1}}

\newcommand{\faces}[1]{\mathcal{F}_{#1}}
\newcommand{\edges}[1]{\mathcal{E}_{#1}}
\newcommand{\vertices}[1]{\mathcal{V}_{#1}}

\newcommand{\FT}{\faces{T}}
\newcommand{\ET}{\edges{T}}
\newcommand{\EF}{\edges{F}}
\newcommand{\VT}{\vertices{T}}
\newcommand{\VF}{\vertices{F}}
\newcommand{\VE}{\vertices{E}}

\newcommand{\Id}{\mathrm{Id}}

\usepackage{pgfplots,pgfplotstable}
\usepackage{tikz-cd}

\graphicspath{{figures/}}

\begin{document}

\title{Commuting quasi-interpolators and Maxwell compactness for a polytopal de Rham complex}

\author[1]{Théophile Chaumont-Frelet}
\author[2,3]{J\'er\^ome Droniou}
\author[1]{Simon Lemaire}
\affil[1]{Inria, Univ.~Lille, CNRS, UMR 8524 -- Laboratoire Paul Painlev\'e, 59000 Lille, France, \email{theophile.chaumont@inria.fr}, \email{simon.lemaire@inria.fr}}
\affil[2]{IMAG, CNRS, Montpellier, France, \email{jerome.droniou@cnrs.fr}}
\affil[3]{School of Mathematics, Monash University, Melbourne, Australia}
\maketitle

\begin{abstract}
  We establish Maxwell compactness results for the Discrete De Rham (DDR) polytopal complex: sequences in this polytopal complex with bounded discrete $\bvec{H}(\CURL)$ (resp.~discrete $\bvec{H}(\DIV)$) norm and orthogonal to discrete gradients (resp.~discrete curls) have $L^2$-relatively compact potential reconstructions. The proof of these results hinges on the design of novel quasi-interpolators, that map the minimal-regularity de Rham spaces onto the discrete DDR spaces and form a commuting diagram. A full set of (primal and adjoint) consistency properties is established for these quasi-interpolators, which paves the way to convergence proofs, under minimal-regularity assumptions, of DDR schemes for partial differential equations based on the de Rham complex. Our analysis is performed with generic mixed boundary conditions, also covering the cases of no boundary conditions or fully homogeneous boundary conditions, and leverages recently introduced liftings from the DDR complex to the continuous de Rham complex.
  \medskip\\
  \textbf{Key words:} Maxwell compactness, polytopal methods, discrete de Rham complex, quasi-interpolators, minimal-regularity.
  \medskip\\
  \textbf{MSC2020.} 65N30, 65N12, 14F40, 35Q61.

\end{abstract}


\section{Introduction}\label{sec:introduction}

This paper brings two major contributions to the theory of polytopal de Rham complexes (arbitrary-order discretisations of the de Rham complex on meshes made of generic polygons/polyhedra). We establish the first Maxwell compactness results for such a discrete complex, and we design the first commuting quasi-interpolators between the minimal-regularity continuous complex and a polytopal complex. We focus our presentation on the Discrete De Rham (DDR) complex \cite{Di-Pietro.Droniou:23}, but our approach does not directly rely on the particulars of this complex and is therefore highly likely applicable to other polytopal complexes.

Compactness is an essential tool in the analysis of partial differential equations (PDEs), especially for nonlinear models or eigenvalue problems. Consequently, it is also critical to the analysis of numerical schemes for PDEs. Discrete versions of the celebrated Rellich--Kondrachov theorem have been known for a long time (see, e.g., \cite{Eymard.Gallouet.ea:00,gal-12-com,Li.ea:15,Droniou.Eymard.ea:18}), but they cover equations whose main operator is the gradient. On the contrary, the driving operators in electromagnetism models are the curl and divergence, for which compactness results -- dubbed Maxwell compactness -- are more challenging to obtain and often require the use of Hodge decompositions and suitable vector potentials; see for example \cite{Weck:74,Weber:80,Jochmann:97,Amrouche.Bernardi.ea:98,bauer2016maxwell,Assous.Ciarlet.ea:18} in the continuous case. Adaptations of Maxwell compactness have been investigated in the discrete setting, see for example \cite{Kikuchi:89,Monk:92,Christiansen:05} for finite element methods.
An adaptation to a non-conforming polytopal method has also been considered in \cite{Lemaire.Pitassi:25} (covering generic boundary conditions) but, contrary to the aforementioned finite element methods or to the polytopal method under study in the present contribution, the hybrid spaces and operators in the latter reference do not (at least presently) form a discrete complex, that is, a compatible discretisation of the de Rham complex
\begin{equation} \label{eq:derham.complex}
  \begin{tikzcd}[column sep=2.5em]
    & 0 \arrow{r}{}
    & \Hgrad{\Omega}\arrow{r}{\GRAD}
    & \Hcurl{\Omega}\arrow{r}{\CURL}
    & \Hdiv{\Omega}\arrow{r}{\DIV}
    &L^2(\Omega)\arrow{r}{}
    & 0
  \end{tikzcd}
\end{equation}
(see Section \ref{ssse:sob} for a precise definition of these classical minimal-regularity Sobolev spaces). Compatible discretisations of this complex are essential to preserve certain properties (e.g., constraints) of PDE models, and/or to obtain discretisations that are robust with respect to certain physical parameters \cite{Pietro.Droniou:22,Beirao-da-Veiga.Dassi.ea:22}. Such discretisations (on standard meshes) include the Lagrange--Nédélec--Raviart-Thomas finite elements, see \cite{Arnold:18} for a generic analysis. More recently, there has been a flourishing activity around the development and analysis of polytopal complexes, that is, discretisations of the de Rham (or other) complex that, contrary to finite element methods, are applicable on meshes made of generic polygonal/polyhedral elements; a non-exhaustive list includes \cite{Beirao-da-Veiga.Brezzi.ea:16,Beirao-da-Veiga.Brezzi.ea:18*1,Beirao-da-Veiga.Brezzi.ea:18*2,Di-Pietro.Droniou.ea:20,Beirao-da-Veiga.Dassi.ea:22}. However, there did not seem to exist, so far, any discrete Maxwell compactness results for such complexes.

\medskip

Interpolation operators play a central role in the numerical analysis of finite elements. The related interpolation errors quantify how well a function with given regularity may be represented in the finite element space. When discretizing coercive linear PDEs with a single unknown, such interpolation errors can be linked together with the discretization error of the scheme through quasi-optimality results, thus showing convergence of the discrete solution, often with optimal rates.
PDEs with multiple unknowns -- such as the Maxwell equations -- are often approximated using mixed finite element methods, whereby each variable is discretized within its own finite element space. The underlying complex structure (e.g., de Rham) is key in the analysis (discrete or not) of such models. To comply with this structure, interpolation operators typically need to satisfy commuting properties, whereby the interpolation operators of the finite element spaces must suitably interact with the corresponding differential operators.

For standard finite element spaces, canonical interpolation operators already satisfy the expected commutation properties, and can often be readily used for the analysis. However, these canonical operators come with an important shortcoming: they are only well-defined for sufficiently smooth functions. For instance, the Lagrange interpolator requires nodal values, which are not well-defined for functions in the minimal-regularity space $\Hgrad{\Omega}$ as soon as the ambient dimension is greater than one. As a by-product, unnecessary regularity assumptions are often required in analyses based on these canonical interpolators. This shortcoming has motivated the development of so-called quasi-interpolation operators, that are able to operate on functions with minimal smoothness requirements.

The first key seminal contributions on the development of quasi-interpolation operators for Lagrange finite elements are due to Cl\'ement~\cite{clement_1975a} and Scott--Zhang~\cite{scott_zhang_1990a}. These early contributions introduce interpolation operators defined for very rough functions, but lacking the commutation properties needed to handle discretisations of the de Rham complex. There has been, since then, a considerable effort in designing suitable commuting quasi-interpolation operators under minimal-regularity assumptions for finite element complexes on simplicial meshes; see, e.g., \cite{schoberl_2001a,Falk_Winth_loc_coch_14,ern_guermond_2016a,Arn_Guz_loc_stab_L2_com_proj_21,ern_gudi_smears_vohralik_2022a,Chaumont-Frelet.Vohralik:24,Chaumont-Frelet.Vohralik:25} and references therein.
Such commuting quasi-interpolators have been instrumental in establishing key results in the numerical analysis of finite element methods, such as discrete Poincar\'e inequalities~\cite{Arnold.Falk.ea:06,ern_guzman_potu_vohralik_2024a}, convergence of multigrid algorithms~\cite{Arn_Falk_Wint_MG_H_div_H_curl_00,schoberl_2001a}, localized orthogonal decomposition~\cite{maalqvist_peterseim_2014a,gallistl_henning_verfurth_2018a}, or duality analysis for time-harmonic Maxwell's equations~\cite{chaumontfrelet_2019a,chaumontfrelet_nicaise_pardo_2018a,ern_guermond_2018a}. To the best of our knowledge, however, no commuting quasi-interpolators had been designed, so far, for polytopal complexes.

\medskip

In this work, we consider the Discrete De Rham (DDR) polytopal complex \cite{Di-Pietro.Droniou.ea:20,Di-Pietro.Droniou:23}. We design commuting quasi-interpolators for this complex, and we prove that it satisfies a discrete version of the Maxwell compactness result. As for finite elements, the canonical DDR interpolators do commute with the corresponding differential operators, but they can only be applied to functions with higher regularity than those in the complex \eqref{eq:derham.complex}. The quasi-interpolators allow us to establish improved primal and adjoint consistency properties of the DDR complex, that operate under minimal-regularity assumptions -- see Theorems~\ref{thm:qI.primal.consistency} and~\ref{thm:qI.adjoint.consistency} below. Such properties are, in turn, instrumental to error analyses of DDR schemes under minimal-regularity assumptions on the solutions (for models which include, for example, rough coefficients).

Additionally -- and, perhaps, more importantly -- these quasi-interpolators are instrumental in the proof of the Maxwell compactness of the DDR complex. To establish this compactness, we indeed use three features of the quasi-interpolators: their applicability to minimal-regularity functions, their consistency properties, and their particular design making them left-inverse of the recently introduced conforming lifting operators of \cite{Di-Pietro.Droniou.ea:25} -- another key element in the proof of the compactness result. 

There is a strong link between the lowest-order DDR complex and Compatible Discrete Operators \cite{Bonelle.Ern:14,Bonelle:14}. As a consequence, our results directly apply to this framework. Moreover, our approach to the construction of quasi-interpolators and to the proof of the discrete Maxwell compactness theorem does not directly rely on the specific design of the Discrete De Rham complex, and is therefore likely to apply to other polytopal complexes such as Virtual Element complexes (upon extending the liftings of \cite{Di-Pietro.Droniou.ea:25} to these methods, which is also probably feasible).

\medskip

The paper is organised as follows. In Section \ref{sec:ddr} we recall some notations (mesh, functional spaces) and briefly present the main elements of the DDR complex, without providing the detailed formulas (not useful to our purpose) for the construction of the discrete operators. Section \ref{sec:compactness} states the main Maxwell compactness results, as well as the more standard Rellich compactness (for the first space in the DDR sequence). The proof of the Maxwell compactness is postponed to Section \ref{sec:proof.curl.compactness}, since it requires the quasi-interpolators that are introduced in Section \ref{sec:quasi-interpolator}, in which we also state the main properties of these interpolators: commutation with differential operators, primal and adjoint consistencies. The proofs of these properties are collected in Section \ref{sec:analysis.qI}. Finally, a brief appendix presents a discrete trace estimate for the DDR version of $\Hgrad{\Omega}$, which is useful to simplify one of the adjoint consistency estimates of the quasi-interpolators.

The entire analysis, including the design and properties of the quasi-interpolators and the proof of the Maxwell compactness, is performed considering generic mixed boundary conditions, which cover as particular cases the situation of fully homogeneous boundary conditions or no boundary conditions (corresponding to natural boundary conditions in practice).

\section{The Discrete De Rham complex}\label{sec:ddr}

\subsection{Mesh}\label{sec:mesh}

We assume that $\Omega$ is a Lipschitz polyhedral open set in $\Real^3$. We consider a polytopal mesh $\Mh$ of $\Omega$ that is part of a regular mesh sequence as per \cite[Definition 1.9]{Di-Pietro.Droniou:20}, and use the orientation notations (briefly recalled hereafter) as in \cite{Di-Pietro.Droniou:23}.
The sets of elements, faces, edges and vertices are respectively denoted by $\Th$, $\Fh$, $\Eh$ and $\Vh$; the generic notation for an element is $T$, for a face $F$, for an edge $E$, and for a vertex $V$. The set of faces $F\subset \partial\Omega$ is denoted by $\Fhb$.
Each $\sfP\in\Th\cup\Fh\cup\Eh$ is assumed to be relatively open in the affine space it spans, and we denote by $h_\sfP$ its diameter; the mesh size is $h\coloneq \max_{T\in\Th}h_T$. For each such $\sfP$, we select $\bvec{x}_\sfP$ such that, in the space it spans, $\sfP$ contains a ball centered at $\bvec{x}_\sfP$ and of radius $\ge c h_{\sfP}$, where $c\in(0,1)$ only depends on the mesh regularity factor. For $V\in\Vh$, $\bvec{x}_V$ denotes the coordinates vector of $V$.

If $T\in\Th$, we let $\FT$ (resp.\ $\ET$, resp.\ $\VT$) be the set of faces (resp.\ edges, resp.\ vertices) belonging to $T$; the notations $\EF$, $\VF$, $\VE$ are defined similarly.
Each $F\in\Fh$ is endowed with a unit normal vector $\normal_F$ that determines its orientation, and if $F\in\FT$ we denote by $\omega_{TF}\in\{-1,1\}$ the relative orientation of $F$ with respect to $T$ -- that is, such that $\omega_{TF}\normal_F$ is pointing out of $T$.
Similarly, each $E\in\Eh$ is equipped with a unit tangent vector $\tangent_E$ and, if $E\in\EF$, $\omega_{FE}\in\{-1,1\}$ is the relative orientation of $E$ with respect to $F$ -- such that $\omega_{FE}\tangent_E$ points in the clockwise direction on $\partial F$ (as determined by the orientation of $F$).

We consider a relatively open Lipschitz subset $\Gamma$ of $\partial\Omega$ on which essential boundary conditions are applied for the continuous and discrete spaces.
We assume that the mesh is compliant with $\Gamma$ in the sense that each face of $\Fh$ is either fully contained in $\Gamma$, or does not intersect $\Gamma$. This ensures that $\Gamma$ is the union of certain (boundary) mesh faces. For future use, we also let $\Gamma^\compl\coloneq\partial\Omega\backslash\overline{\Gamma}$ denote the (interior of the) complement of $\Gamma$ in $\partial\Omega$.

The mesh regularity assumption ensures the existence of a \emph{matching simplicial submesh} $\MSh$ of $\Mh$ \cite[Definition 1.7]{Di-Pietro.Droniou:20},
whose shape regularity factor is bounded below by the regularity factor of $\Mh$. $\Sh$ denotes the set of (3-dimensional) simplices in $\MSh$ and we will use the letter $S$ to denote such a generic simplex.
For all $T\in\Th$, the set of simplices of $\Sh$ contained in $T$ is denoted by $\Sh(T)$; by mesh regularity, the cardinal of $\Sh(T)$ is uniformly bounded above, and the diameter of each $S \in \Sh(T)$ is bounded below by $c h_T$, where $c\in(0,1)$ only depends on the mesh regularity factor.

The last mesh concept we will require is that of \emph{polytopal patch} around $T\in\Th$, defined by
\begin{equation}\label{eq:def.ppatch}
\ppatch{T}=\{T'\in\Th\,:\,\overline{T}\cap\overline{T'}\not=\emptyset\}.
\end{equation}
We will, from time to time, identify patches with the interior of the closure of the domain $\cup_{T'\in\ppatch{T}}T'$ they cover. So, for example, we write $L^2(\ppatch{T})$ for $L^2\big(\mathrm{int}(\overline{\cup_{T'\in\ppatch{T}}T'})\big)$. The same abuse of notation is made in norms, or for Sobolev spaces.

\subsection{Functional spaces}\label{sec:functional.spaces}

To avoid the proliferation of constants, we write $a\lesssim b$ for $a\le Cb$ with $C$ not depending on $h$, $a$, $b$ or the mesh element considered, but possibly depending on $\Omega$, $\Gamma$, the mesh regularity factor, the polynomial degrees involved in $a$ or $b$, and possibly additional quantities  that will be specified when required (such additional quantities could be physical parameters; see, in particular, Section \ref{sec:ddr.norms}). The notation $a \eqsim b$ stands for $a\lesssim b$ and $b\lesssim a$. 

\subsubsection{Sobolev spaces} \label{ssse:sob}

For $\omega$ an open set (e.g., $\omega=\Omega$ or $\omega=T\in\Th$), the Sobolev spaces associated with $\GRAD$, $\CURL$ and $\DIV$ are
\begin{alignat*}{2}
    \Hgrad{\omega}&\coloneq \big\{q\in L^2(\omega)\,:\,\GRAD q\in L^2(\omega)^3\big\},\\
    \Hcurl{\omega}&\coloneq \big\{\bvec{v}\in L^2(\omega)^3\,:\,\CURL\bvec{v}\in L^2(\omega)^3\big\},\\
    \Hdiv{\omega}&\coloneq \big\{\bvec{w}\in L^2(\omega)^3\,:\,\DIV\bvec{w}\in L^2(\omega)\big\}.
\end{alignat*}
If $\Sigma$ is a relatively open Lipschitz subset of $\partial\omega$ (in practice, $\Sigma=\Gamma\cap\partial\omega$ or $\Sigma=\Gamma^\compl\cap\partial\omega$), the subspaces with zero relevant trace on $\Sigma$ are denoted by an index:
\begin{equation}\label{eq:def.Hspaces}
\begin{aligned}
   \HgradS{\omega}&\coloneq \big\{q\in \Hgrad{\omega}\,:\,q|_\Sigma=0\big\},\\
    \HcurlS{\omega}&\coloneq \big\{\bvec{v}\in \Hcurl{\omega}\,:\,\bvec{v}_{{\rm t},\Sigma}=\bvec{0}\big\},\\
    \HdivS{\omega}&\coloneq \big\{\bvec{w}\in \Hdiv{\omega}\,:\, w_{{\rm n},\Sigma}=0\big\}.
\end{aligned}
\end{equation}
Here and in the following, if $\bvec{\xi}$ is a (sufficiently smooth) vector field and $X$ is a two-dimensional set with unit normal $\normal_X$ (determined by the local context -- e.g., for a mesh face $F$ this would be the fixed normal $\normal_F$, while for $\Gamma$ it is the outer normal to $\Omega$), the tangential and normal traces of $\bvec{\xi}$ are denoted by
\[
\bvec{\xi}_{{\rm t},X}\coloneq \normal_X\times (\bvec{\xi}|_X\times\normal_X)\,,\qquad\xi_{{\rm n},X}\coloneq \bvec{\xi}|_X\cdot\normal_X.
\]
For vector fields with minimal-regularity in~\eqref{eq:def.Hspaces}, traces on part of the boundary are rigorously defined, e.g., in~\cite{fernandes_gilardi_1997a}.

To unify some expressions, we will often use the notation $\Hgen{\omega}$, or its variant $\HgenS{\omega}$, with $\gensymbol$ being any differential operator $\GRAD$, $\CURL$ or $\DIV$. The use of this generic notation to handle different spaces and operators is done at the expense, in some places, of minor abuses of notation (for example, scalar- and vector-valued spaces or operators are no longer distinguished using boldface).

For $X$ a measurable set, the norms on $L^2(X)$ and its vector-valued version $L^2(X)^3$ are irrespectively denoted by $\norm{L^2(X)}{{\cdot}}$.

\subsubsection{Local polynomial and finite element spaces}

For $\ell\ge 0$, the space of polynomial functions of total degree $\le \ell$ on $\sfP\in \Th\cup\Fh\cup\Eh\cup\Sh$ is denoted by $\Poly{\ell}(\sfP)$, with the convention
that $\Poly{-1}(\sfP)=\{0\}$. We use the boldface notation $\vPoly{\ell}(\sfP)$ for vector-valued polynomials $\sfP\to\Real^3$ and, when $\sfP\in\Fh$ is a mesh face, we restrict these vector-valued polynomials to the ones that are tangent to $\sfP$.
The $L^2(\sfP)$-orthogonal projector onto $\Poly{\ell}(\sfP)$, or its vector-valued version, is denoted by $\lproj{\ell}{\sfP}$. We will also write $\Poly{\ell}(\Th)$ (resp.\ $\Poly{\ell}(\Sh)$) for the broken polynomial space on the polytopal mesh $\Th$ (resp.\ the simplicial submesh $\Sh$). The projector on $\Poly{\ell}(\Th)$, obtained patching the local projectors $(\lproj{\ell}{T})_{T\in\Th}$, is denoted by $\lproj{\ell}{\Th}$.

\medskip

Still considering $\ell\ge 0$, the local Lagrange, Nédélec and Raviart--Thomas spaces on a polytopal element or simplex $T\in\Th\cup\Sh$, or a face $F\in\Fh$, are respectively denoted by
\begin{align*}
\Lag{\ell}{T}&\coloneq \Poly{\ell}(T),\\
\NE{\ell}{T}&\coloneq \vPoly{\ell-1}(T) + (\bvec{x}-\bvec{x}_T)\times \vPoly{\ell-1}(T),\\
\RT{\ell}{\sfP}&\coloneq \vPoly{\ell-1}(\sfP) + (\bvec{x}-\bvec{x}_\sfP) \Poly{\ell-1}(\sfP)\quad\text{ for $\sfP=T,F$}.
\end{align*}
We note that, for $\ell=0$, the last two spaces are reduced to $\{\bvec{0}\}$.
Whenever useful, for $\gensymbol\in\{\Lagsymbol,\NEsymbol,\RTsymbol\}$, these spaces are irrespectively denoted by $\FEgen{\ell}{\sfP}$.
The (scalar or vector) $L^2(\sfP)$-orthogonal projector onto $\FEgen{\ell}{\sfP}$ is denoted by $\genlproj{\ell}{\sfP}$.

\medskip

Taking $\gensymbol\in\{\Lagsymbol,\NEsymbol,\RTsymbol\}$, the global conforming finite element space on the simplicial mesh $\Sh$ associated with $\gensymbol$, and its version with homogeneous boundary conditions on $\Gamma$, are respectively denoted by
\begin{align*}
\FEgen{\ell}{\Sh}&\coloneq\{u\in\Hgen{\Omega}\,:\,u|_{S}\in\FEgen{\ell}{S}\quad\forall S\in\Sh\},\\
\FEgenG{\ell}{\Sh}&\coloneq \FEgen{\ell}{\Sh}\cap \HgenG{\Omega}.
\end{align*}
We will also use the notation $\FEgen{\ell}{\Sh(T)}$ (resp.~$\FEgenG{\ell}{\Sh(T)}$) for the restriction to $\Sh(T)$ of $\FEgen{\ell}{\Sh}$ (resp.~$\FEgenG{\ell}{\Sh}$).

\subsection{The DDR complex}

We give here a brief presentation of the main elements of the Discrete De Rham complex, sometimes omitting precise formulas that are not required for an overall vision of the complex; wherever necessary, we will recall the relevant properties of the various DDR operators.
We refer the reader to \cite{Di-Pietro.Droniou:23} for all the details, as well as \cite{Bonaldi.Di-Pietro.ea:25} for a synthetic presentation of all the spaces and operators, in any dimension, using the exterior calculus framework; we elected to not adopt the latter presentation to make our exposition accessible to readers without knowledge in exterior calculus.

Let an integer $k\ge 0$ be given. The Discrete De Rham complex reads
\begin{equation}\label{eq:ddr.complex}
  \begin{tikzcd}[column sep=2.5em]
    & 0 \arrow{r}{}
    & \Xgrad{k}{\Th}\arrow{r}{\uG{k}{h}}
    & \Xcurl{k}{\Th}\arrow{r}{\uC{k}{h}}
    & \Xdiv{k}{\Th}\arrow{r}{\D{k}{h}}
    &\Poly{k}(\Th)\arrow{r}{}
    & 0
  \end{tikzcd}
\end{equation}
with discrete spaces defined by
\begin{equation*} 
\begin{aligned}
\Xgrad{k}{\Th}\coloneq \big\{\dof{q}{h}={}&\big((q_T)_{T\in\Th},(q_F)_{F\in\Fh},(q_E)_{E\in\Eh},(q_V)_{V\in\Vh}\big),\\
&q_T\in\Poly{k-1}(T)\quad \forall T\in\Th\,,\qquad
q_F\in\Poly{k-1}(F)\quad \forall F\in\Fh\,,\\
&q_E\in\Poly{k-1}(E)\quad \forall E\in\Eh\,,\qquad
q_V\in\Real\quad \forall V\in\Vh
\big\},
\end{aligned}
\end{equation*}
\begin{equation*} 
\begin{aligned}
\Xcurl{k}{\Th}\coloneq \big\{\vdof{v}{h}={}&\big((\bvec{v}_T)_{T\in\Th},(\bvec{v}_F)_{F\in\Fh},(v_E)_{E\in\Eh}\big),\\
&\bvec{v}_T\in\RT{k}{T}\quad \forall T\in\Th\,,\qquad
\bvec{v}_F\in\RT{k}{F}\quad \forall F\in\Fh\,,\\
&v_E\in\Poly{k}(E)\quad \forall E\in\Eh
\big\},
\end{aligned}
\end{equation*}
and
\begin{equation*} 
\begin{aligned}
\Xdiv{k}{\Th}\coloneq \big\{\vdof{w}{h}={}&\big((\bvec{w}_T)_{T\in\Th},(w_F)_{F\in\Fh}\big),\\
&\bvec{w}_T\in\NE{k}{T}\quad \forall T\in\Th\,,\qquad
w_F\in\Poly{k}(F)\quad \forall F\in\Fh
\big\}.
\end{aligned}
\end{equation*}
The discrete differential operators $\uG{k}{h}$, $\uC{k}{h}$ and $\D{k}{h}$ are defined from the unknowns in these spaces by mimicking the integration-by-parts formulas involving $\GRAD$, $\CURL$ and $\DIV$ \cite[Section 3.3]{Di-Pietro.Droniou:23}. Their specific precise formulas will not be useful to our analysis; we will only recall their useful properties when needed.

The (standard) interpolators on the DDR spaces are obtained by applying appropriate $L^2$-orthogonal projectors onto the component polynomial spaces:
\begin{equation}\label{eq:interpolators}
\begin{aligned}
\Igrad{k}{h}q&\coloneq \big((\lproj{k-1}{T}q)_{T\in\Th}, (\lproj{k-1}{F}q)_{F\in\Fh}, (\lproj{k-1}{E}q)_{E\in\Eh}, (q(\bvec{x}_V))_{V\in\Vh}\big) &&\quad\forall q\in C^0(\overline{\Omega}),\\
\Icurl{k}{h}\bvec{v}&\coloneq \big((\RTproj{k}{T}\bvec{v})_{T\in\Th}, (\RTproj{k}{F}\bvec{v}_{\mathrm{t},F})_{F\in\Fh}, (\lproj{k}{E}(\bvec{v}\cdot\tangent_E))_{E\in\Eh}\big) &&\quad\forall \bvec{v}\in C^0(\overline{\Omega})^3,\\
\Idiv{k}{h}\bvec{w}&\coloneq \big((\NEproj{k}{T}\bvec{w})_{T\in\Th}, (\lproj{k}{F}w_{{\rm n},F})_{F\in\Fh}\big) &&\quad\forall \bvec{w}\in C^0(\overline{\Omega})^3.
\end{aligned}
\end{equation}
These interpolators commute with the differential operators $\GRAD$, $\CURL$ and $\DIV$ but require, to be well-defined, a certain level of continuity of their argument. We wrote \eqref{eq:interpolators} with continuous functions, but weaker regularity assumptions can be considered; specifically, each interpolator can be applied to functions with suitable single-valued $L^1$-traces on specific mesh entities (vertices, edges, faces, or elements). However, even these weaker assumptions are beyond the basic regularity of the underlying energy spaces $\Hgrad{\Omega}$, $\Hcurl{\Omega}$ or $\Hdiv{\Omega}$. The quasi-interpolators defined in Section \ref{sec:quasi-interpolator} will remedy this restriction and provide interpolators defined on these energy spaces, and that commute with the differential operators.

\medskip

To treat all the DDR spaces at once, we will often use the notation $\Xgen{k}{\Th}$ with $\gensymbol$ standing for $\GRAD$, $\CURL$ or $\DIV$.
The restriction to $\sfP\in\Th\cup\Fh\cup \Eh$ of a vector $\dof{z}{h}\in\Xgen{k}{\Th}$ (that is, the vector gathering the components of $\dof{z}{h}$ on the elements/faces/edges/vertices belonging to $\sfP$) is denoted by $\dof{z}{\sfP}$. Likewise, we write $\Xgen{k}{\sfP}$ for the space of such restrictions.
Similar local notations apply to the interpolators and the discrete differential operators.

\medskip

Each DDR space above is equipped with the following potential reconstructions: for $E\in\Eh$, $F\in\Fh$ or $T\in\Th$,
\begin{alignat*}{3}
&\tr{k+1}{E}:\Xgrad{k}{E}\to \Poly{k+1}(E)\,,&&\qquad \tr{k+1}{F}:\Xgrad{k}{F}\to\Poly{k+1}(F)\,,&&\qquad\Pgrad{k+1}{T}:\Xgrad{k}{T}\to\Poly{k+1}(T),\\
&&&\qquad\trt{k}{F}:\Xcurl{k}{F}\to \vPoly{k}(F)\,,&&\qquad \Pcurl{k}{T}:\Xcurl{k}{T}\to\vPoly{k}(T),\\
&&&&&\qquad\Pdiv{k}{T}: \Xdiv{k}{T} \to\vPoly{k}(T).
\end{alignat*}
The role of these reconstructions is to provide piecewise (discontinuous) polynomials that play the role of discrete functions -- e.g.\ for defining consistent $L^2$-like inner products, see Section \ref{sec:ddr.norms} below. The potentials defined on the gradient space play the role of scalar traces (on edges, faces) or scalar potentials in elements, those on the curl space play the role of tangential traces on faces or vector potentials in the elements, and the reconstructions on the divergence space play the role of vector potentials in the elements.

To distinguish the degrees of the reconstructions in the gradient space and in the curl and divergence spaces, we set
\begin{equation}\label{eq:def.kgen}
k_{\GRAD}\coloneq k+1\,,\quad k_{\CURL}=k_{\DIV}=k.
\end{equation}
Consistently with the unified notation adopted so far, we then use $\Pgen{k}{T}$ to denote any of $\Pgrad{k+1}{T}$, $\Pcurl{k}{T}$ or $\Pdiv{k}{T}$. The notation $\Pgen{k}{h}$ refers to the piecewise polynomial reconstruction on $\Th$ obtained by patching together the local reconstructions:
\[
(\Pgen{k}{h}\dof{z}{h})|_T=\Pgen{k}{T}\dof{z}{T}\qquad\forall T\in\Th\,,\ \forall \dof{z}{h}\in\Xgen{k}{\Th}.
\]

The spaces with homogeneous boundary values on $\Gamma$ are denoted by an index, as for the Sobolev spaces in \eqref{eq:def.Hspaces} (for $\Sigma=\Gamma$):
\begin{equation}\label{eq:def.XDDR.D}
\begin{aligned}
\XgradG{k}{\Th}&\coloneq\big\{\dof{q}{h}\in\Xgrad{k}{\Th}\,:\,\text{$q_V=0$ for all $V\in\overline{\Gamma}$, 
$q_E=0$ for all $E\subset \overline{\Gamma}$,}\\
&\qquad\text{ and $q_F=0$ for all $F\subset\Gamma$}\big\},\\
\XcurlG{k}{\Th}&\coloneq\big\{\vdof{v}{h}\in\Xcurl{k}{\Th}\,:\,\text{$v_E=0$ for all $E\subset \overline{\Gamma}$ and $\bvec{v}_F= \bvec{0}$ for all $F\subset\Gamma$}\big\},\\
\XdivG{k}{\Th}&\coloneq\big\{\vdof{w}{h}\in\Xdiv{k}{\Th}\,:\,\text{$w_F=0$ for all $F\subset\Gamma$}\big\}.
\end{aligned}
\end{equation}
From the definitions in \cite{Di-Pietro.Droniou:23} of the discrete differential operators and potential reconstructions, and the compliance of the mesh with $\Gamma$, we easily see that:
\begin{itemize}
\item The spaces \eqref{eq:def.XDDR.D} form a subcomplex of \eqref{eq:ddr.complex} (the space $\Poly{k}(\Th)$ being left unchanged).
\item For all $E\subset \overline{\Gamma}$ and $F\subset\Gamma$, the traces $\tr{k+1}{E}$, $\tr{k+1}{F}$ and $\trt{k}{F}$ vanish when
applied to vectors in the relevant spaces \eqref{eq:def.XDDR.D}.
\item Restricted to $C^0(\overline{\Omega})\cap \HgradG{\Omega}$, $C^0(\overline{\Omega})^3\cap \HcurlG{\Omega}$ or $C^0(\overline{\Omega})^3\cap \HdivG{\Omega}$, the interpolators \eqref{eq:interpolators} have the spaces \eqref{eq:def.XDDR.D} as co-domains.
\end{itemize}

\subsection{Components and $L^2$-like norms}\label{sec:ddr.norms}

On each DDR space,we define a components norm and an $L^2$-like inner product.
The components norm of $\dof{z}{h}\in \Xgen{k}{\Th}$ is defined by
\begin{equation}\label{eq:def.component.norm}
\begin{aligned}
\tnorm{\gensymbol,h}{\dof{z}{h}}&\coloneq \left(\sum_{T\in\Th}\tnorm{\gensymbol,T}{\dof{z}{T}}^2\right)^{1/2}\\
\text{ with }
\tnorm{\gensymbol,T}{\dof{z}{T}}^2&\coloneq\norm{L^2(T)}{z_T}^2+ h_T\sum_{F\in\FT}\tnorm{\gensymbol,F}{\dof{z}{F}}^2,\\
\text{ and }
\tnorm{\gensymbol,F}{\dof{z}{F}}^2&\coloneq\norm{L^2(F)}{z_F}^2+ h_F\sum_{E\in\EF}\norm{L^2(E)}{z_E}^2+h_F^2\sum_{V\in\VF}|z_V|^2,
\end{aligned}
\end{equation}
where, of course, the sums are actually only taken on the components that appear in $\Xgen{k}{\Th}$ (so, for example, only faces and edges for $\Xcurl{k}{\Th}$). 

We introduce a function $\mu$, which may represent physical parameters of interest for models discretized by the DDR method (e.g.,  mobility for porous media flows, permittivity/permeability for electromagnetic equations, etc.). This parameter will be involved in the definition of the discrete $L^2$-like inner products, and thus in the statement of their properties.

\begin{assumption}[Physical parameter]\label{assum:physical.parameter}
Let $n_\gensymbol=1$ if $\gensymbol=\GRAD$ and $n_\gensymbol=3\times 3$ if $\gensymbol\in\{\CURL,\DIV\}$, and recall the definition \eqref{eq:def.kgen} of $k_\gensymbol$. The function
$\mu:\Omega\to \Real^{n_\gensymbol}$ is a piecewise-$C^{k_\gensymbol}$ function on $\Th$ that is symmetric, uniformly bounded and coercive.
From hereon, we assume that the hidden constants in $\lesssim$ (see Section \ref{sec:functional.spaces}) can also depend on the coercivity constant of $\mu$, as well as on the bounds on $\mu$ and its local derivatives up to order $k_\gensymbol$.
\end{assumption}

We will need to consider Sobolev spaces associated with such physical parameters. Re-using the notation from Section~\ref{ssse:sob}, for $\gensymbol\in\{\GRAD,\CURL,\DIV\}$, we set
\begin{equation*} 
\HgenSmu{\Omega}\coloneq\{z\in L^2(\Omega)\,:\,\mu z\in \HgenS{\Omega}\}\quad\text{and}\quad
\Hgenmu{\Omega} \coloneq \Hgenemptymu{\Omega}.
\end{equation*}

Assumption \ref{assum:physical.parameter} allows us to define the $\mu$-weighted $L^2$-like inner product of $\dof{z}{h},\dof{y}{h}\in\Xgen{k}{\Th}$ by
\begin{equation}\label{eq:def.L2.inner}
\begin{aligned}
(\dof{z}{h},\dof{y}{h})_{\mu,\gensymbol,h}&\coloneq \sum_{T\in\Th}(\dof{z}{T},\dof{y}{T})_{\mu,\gensymbol,T}\\
\text{ with }
(\dof{z}{T},\dof{y}{T})_{\mu,\gensymbol,T}&\coloneq \int_T \mu\Pgen{k}{T}\dof{z}{T}\cdot\Pgen{k}{T}\dof{y}{T} + \overline{\mu}_T s_{\gensymbol,T}(\dof{z}{T},\dof{y}{T}),
\end{aligned}
\end{equation}
where $\overline{\mu}_T>0$ is a scalar quantity that represents the magnitude of $\mu|_T$ (for example, the average of its scalar or matrix norm).
The norm associated with this inner product is
\begin{equation}\label{eq:def.L2.norm}
\begin{aligned}
\norm{\mu,\gensymbol,h}{\dof{z}{h}}&\coloneq\left(\sum_{T\in\Th}\norm{\mu,\gensymbol,T}{\dof{z}{T}}^2\right)^{1/2}\\
\text{ with }
\norm{\mu,\gensymbol,T}{\dof{z}{T}}^2&\coloneq\norm{L^2(T)}{\mu^{1/2}\Pgen{k}{T}\dof{z}{T}}^2 + \overline{\mu}_Ts_{\gensymbol,T}(\dof{z}{T},\dof{z}{T}).
\end{aligned}
\end{equation}
In the definitions above, $s_{\gensymbol,T}$ is a semi-definite positive bilinear form penalising the differences between the potential reconstruction in $T$ and those on the faces or edges of $T$; see, e.g., \cite[Section 4.4]{Di-Pietro.Droniou:23}.
The main properties of $s_{\gensymbol,T}$ are (i) its polynomial consistency (recall the definition \eqref{eq:def.kgen} of $k_\gensymbol$): 
\begin{equation}\label{eq:polynomial.consistency.sT}
s_{\gensymbol,T}(\Igen{k}{T}w,\dof{z}{T})=0\qquad\forall w\in\Poly{k_\gensymbol}(T)\,,\quad\forall \dof{z}{T}\in\Xgen{k}{T},
\end{equation}
and (ii) the norm equivalence property it ensures:
\begin{equation}\label{eq:coercivity.sT}
  \norm{\mu,\gensymbol,T}{\dof{z}{T}} \eqsim \tnorm{\gensymbol,T}{\dof{z}{T}}\qquad\forall \dof{z}{T}\in\Xgen{k}{T}.
\end{equation}
We note that, combining \eqref{eq:def.L2.norm} and \eqref{eq:coercivity.sT}, we have
\begin{equation}\label{eq:bound.Pgen.sTgen}
\norm{L^2(T)}{\Pgen{k}{T}\dof{z}{T}}+s_{\gensymbol,T}(\dof{z}{T},\dof{z}{T})^{1/2}\lesssim \tnorm{\gensymbol,T}{\dof{z}{T}}\qquad\forall \dof{z}{T}\in\Xgen{k}{T}.
\end{equation}
We will also need the global stabilisation bilinear form $s_{\gensymbol,h}$ on $\Xgen{k}{\Th}$, assembled from the local ones:
\[
s_{\gensymbol,h}(\dof{z}{h},\dof{y}{h})\coloneq \sum_{T\in\Th}s_{\gensymbol,T}(\dof{z}{T},\dof{y}{T})\qquad\forall\dof{z}{h},\dof{y}{h}\in\Xgen{k}{\Th}.
\]

\section{Compactness results for the DDR complex}\label{sec:compactness}

In the following, we assume that $(\Mh)_{h\in\mathcal H}$ is a family of regular polytopal meshes, with $0$ as unique accumulation point of $\mathcal H$.

The main results in this section are the first two theorems. Theorem \ref{thm:curl.compactness} is a discrete version of the following Maxwell compactness result (stated here, for simplicity, with $\mu\equiv 1$): if a sequence is bounded in $\HcurlG{\Omega}$ and has zero divergence and zero normal trace on $\Gamma^\compl$ (i.e., is orthogonal to gradients of functions in $\HgradG{\Omega}$), then it is relatively compact in $L^2(\Omega)^3$. The second result, Theorem \ref{thm:div.compactness}, is the pendant for $\HdivG{\Omega}$: a sequence bounded in this space and with zero curl and tangential trace on $\Gamma^\compl$ (i.e., that is orthogonal to curls of fields in $\HcurlG{\Omega}$) is relatively compact in $L^2(\Omega)^3$. The proofs of these results rely on the quasi-interpolators introduced in Section \ref{sec:quasi-interpolator}, as well as on well-designed conforming liftings for the DDR complex \cite{Di-Pietro.Droniou.ea:25}.

\begin{theorem}[Maxwell compactness (I)]\label{thm:curl.compactness}
Consider $\mu$ that satisfies Assumption \ref{assum:physical.parameter}.
Let $(\vdof{v}{h})_{h\in\mathcal H}$ be such that $\vdof{v}{h}\in\XcurlG{k}{\Th}$ for each $h\in\mathcal H$, and
\begin{alignat}{2}
\label{eq:compact.bound}
&\text{$\big(\tnorm{\CURL,h}{\vdof{v}{h}}+\tnorm{\DIV,h}{\uC{k}{h}\vdof{v}{h}}\big)_{h\in\mathcal H}$ is bounded},\\
\label{eq:compact.orth}
&(\vdof{v}{h},\uG{k}{h}\dof{q}{h})_{\mu,\CURL,h}=0\qquad\forall\dof{q}{h}\in\XgradG{k}{\Th}\,,\quad\forall h\in\mathcal H.
\end{alignat}
Then, there exists $\bvec{v}\in\HcurlG{\Omega}\cap\HdivGcmu{\Omega}$ such that $\DIV (\mu\bvec{v}) \equiv 0$ and, up to a subsequence as $h\to 0$,
$\Pcurl{k}{h}\vdof{v}{h}\to \bvec{v}$ in $L^2(\Omega)^3$ and $s_{\CURL,h}(\vdof{v}{h},\vdof{v}{h})\to 0$.
\end{theorem}

\begin{proof}
See Section \ref{sec:proof.curl.compactness}.
\end{proof}

\begin{theorem}[Maxwell compactness (II)]\label{thm:div.compactness}
Consider $\mu$ that satisfies Assumption \ref{assum:physical.parameter}.
Let $(\vdof{w}{h})_{h\in\mathcal H}$ be such that $\vdof{w}{h}\in\XdivG{k}{\Th}$ for each $h\in\mathcal H$, and
\begin{alignat*}{2}
&\text{$\big(\tnorm{\DIV,h}{\vdof{w}{h}}+\norm{L^2(\Omega)}{\D{k}{h}\vdof{w}{h}}\big)_{h\in\mathcal H}$ is bounded},\\
&(\vdof{w}{h},\uC{k}{h}\vdof{v}{h})_{\mu,\DIV,h}=0\qquad\forall \vdof{v}{h} \in\XcurlG{k}{\Th}\,,\quad\forall h\in\mathcal H.
\end{alignat*}
Then, there exists $\bvec{w}\in\HdivG{\Omega}\cap\HcurlGcmu{\Omega}$ such that $\CURL (\mu\bvec{w}) \equiv\bvec{0}$ and, up to a subsequence as $h\to 0$,
$\Pdiv{k}{h}\vdof{w}{h}\to \bvec{w}$ in $L^2(\Omega)^3$ and $s_{\DIV,h}(\vdof{w}{h},\vdof{w}{h})\to 0$.
\end{theorem}

\begin{proof}
The proof is similar to that of Theorem \ref{thm:curl.compactness} and is therefore left to the reader.
\end{proof}

Our last compactness result is a discrete version of the Rellich theorem, on the compactness of sequences that are bounded in $\HgradG{\Omega}$. 
Discrete Rellich theorems have been studied for a long time (see, e.g., \cite{Eymard.Gallouet.ea:00,Droniou.Eymard.ea:18}), and are
easier to establish than Maxwell compactness results. The proof of the following proposition is only provided for the sake of completeness.

\begin{theorem}[Discrete Rellich theorem]\label{thm:grad.compactness}
Let $(\dof{q}{h})_{h\in\mathcal H}$ be such that $\dof{q}{h}\in\XgradG{k}{\Th}$ for each $h\in\mathcal H$, and $(\tnorm{\GRAD,h}{\dof{q}{h}}+\tnorm{\CURL,h}{\uG{k}{h}\dof{q}{h}})_{h\in\mathcal H}$ is bounded.
Then, there exists $q\in\HgradG{\Omega}$ such that, up to a subsequence as $h\to 0$,
$\Pgrad{k+1}{h}\dof{q}{h}\to q$ in $L^r(\Omega)$ for all $r<6$, and $s_{\GRAD,h}(\dof{q}{h},\dof{q}{h})\to 0$.
\end{theorem}

\begin{proof}
Consider the vector $\dof{z}{h}=((z_T)_{T\in\Th},(z_F)_{F\in\Fh})\coloneq\big((\Pgrad{k+1}{T}\dof{q}{T})_{T\in\Th},(\tr{k+1}{F}\dof{q}{F})_{F\in\Fh}\big)$ in the equal-order HHO space of degree $k+1$, and recall the discrete HHO norm of $\dof{z}{h}$:
\[
\norm{1,h}{\dof{z}{h}}^2\coloneq \sum_{T\in\Th}\left(\norm{L^2(T)}{\GRAD \Pgrad{k+1}{T}\dof{q}{T}}^2+\sum_{F\in\FT}h_F^{-1}\norm{L^2(F)}{\Pgrad{k+1}{T}\dof{q}{T}-\tr{k+1}{F}\dof{q}{F}}^2\right).
\]
By \cite[Lemma 7]{Di-Pietro.Droniou:23}, we have $\norm{1,h}{\dof{z}{h}}\lesssim \tnorm{\CURL,h}{\uG{k}{h}\dof{q}{h}}$, and $(\norm{1,h}{\dof{z}{h}})_{h\in\mathcal H}$ is therefore bounded. Noting that $z_F=\tr{k+1}{F}\dof{q}{F}\equiv 0$ whenever $F\subset\Gamma$ (since $\dof{q}{h}\in\XgradG{k}{\Th}$), and that $|\sum_{T\in\Th}\int_T z_T|\lesssim \tnorm{\GRAD,h}{\dof{q}{h}}\lesssim 1$, we can invoke \cite[Theorem 6.8]{Di-Pietro.Droniou:20} to get $q\in\HgradG{\Omega}$ such that, up to a subsequence, $\Pgrad{k+1}{h}\dof{q}{h}\to q$ in $L^r(\Omega)$ for all $r<6$.
The convergence to $0$ of the stabilisation also easily follows from the estimates in \cite[Lemma 7]{Di-Pietro.Droniou:23}, which show that $s_{\GRAD,h}(\dof{q}{h},\dof{q}{h})\lesssim h^2 \tnorm{\CURL,h}{\uG{k}{h}\dof{q}{h}}^{2}$.
\end{proof}

\begin{remark}[$L^2$-bound assumptions] \label{rem:L2}
In some cases, upon strengthening the orthogonality conditions, discrete Poincar\'e inequalities enable the removal of the $L^2$-bound assumptions on   $\vdof{v}{h}$, $\vdof{w}{h}$ and $\dof{q}{h}$ in Theorems \ref{thm:curl.compactness}, \ref{thm:div.compactness} and  \ref{thm:grad.compactness}.
For example, if $\Gamma=\emptyset$, Theorem \ref{thm:curl.compactness} holds if we replace \eqref{eq:compact.bound}--\eqref{eq:compact.orth} by
\begin{alignat}{2}
\label{eq:compact.bound.2}
&\text{$(\tnorm{\DIV,h}{\uC{k}{h}\vdof{v}{h}})_{h\in\mathcal H}$ is bounded},\\
\label{eq:compact.orth.2}
&(\vdof{v}{h},\vdof{z}{h})_{\mu,\CURL,h}=0\qquad\forall \vdof{z}{h}\in\ker\uC{k}{h}\,,\quad\forall h\in\mathcal H.
\end{alignat}
Indeed, \eqref{eq:compact.orth.2} implies \eqref{eq:compact.orth} (since $\Image\uG{k}{h}\subset \ker\uC{k}{h}$ by complex property) and, by the discrete Poincar\'e inequalities of \cite[Theorem 2]{Di-Pietro.Hanot:24} or \cite[Corollary 5]{Di-Pietro.Droniou.ea:25*1}, the assumption
\eqref{eq:compact.orth.2} implies that $\tnorm{\CURL,h}{\vdof{v}{h}}\lesssim \tnorm{\DIV,h}{\uC{k}{h}\vdof{v}{h}}$ which, together with \eqref{eq:compact.bound.2}, yields \eqref{eq:compact.bound}. 

We note that the assumptions \eqref{eq:compact.bound.2}--\eqref{eq:compact.orth.2} correspond to the ones chosen for the Maxwell compactness results of the non-conforming polytopal method in \cite{Lemaire.Pitassi:25} (in essence, these assumptions propose a different handling of the cohomology than \eqref{eq:compact.bound}--\eqref{eq:compact.orth}).
\end{remark}

\section{Quasi-interpolators for the DDR complex}\label{sec:quasi-interpolator}

In this section, we design quasi-interpolators for the DDR sequence. These are, contrary to \eqref{eq:interpolators}, interpolators with minimal-regularity spaces as domains; they therefore allow for error analyses of schemes without assuming smoothness of the solution to the continuous model. These quasi-interpolators also commute with the differential operators in the de Rham/DDR sequences, and enjoy (primal and adjoint) consistency properties, which are key to using them for error estimation.

\subsection{Design and cochain property}\label{sec:prop.lift}

We define here the DDR quasi-interpolators and state their cochain property.
Let $\ell\ge k+1$ and 
\begin{align*}
&\qILag{\ell}{h}:\HgradG{\Omega}\to\LagG{\ell}{\Sh}\,,\quad
\qINE{\ell}{h}:\HcurlG{\Omega}\to\NEG{\ell}{\Sh},\\
&\text{and }\quad
\qIRT{\ell}{h}:\HdivG{\Omega}\to\RTG{\ell}{\Sh}
\end{align*}
be the linear quasi-interpolators, defined in
\cite{Chaumont-Frelet.Vohralik:24,Chaumont-Frelet.Vohralik:25,ern_gudi_smears_vohralik_2022a},
for the finite element sequence of degree $\ell$ on the simplicial submesh $\Sh$. 
The DDR (linear) quasi-interpolators are then defined by
\begin{subequations}\label{eq:def.qI}
\begin{alignat}{2}
\label{eq:def.qIgrad}
\qIgrad{k}{h}&\coloneq \Igrad{k}{h}\circ\qILag{\ell}{h}:\HgradG{\Omega}\to\XgradG{k}{\Th}\,,\\
\label{eq:def.qIcurl}
\qIcurl{k}{h}&\coloneq \Icurl{k}{h}\circ\qINE{\ell}{h}:\HcurlG{\Omega}\to\XcurlG{k}{\Th}\,,\\
\label{eq:def.qIdiv}
\qIdiv{k}{h}&\coloneq \Idiv{k}{h}\circ\qIRT{\ell}{h}:\HdivG{\Omega}\to\XdivG{k}{\Th}.
\end{alignat}
\end{subequations}

The finite element quasi-interpolators $\qILag{\ell}{h}$, $\qINE{\ell}{h}$ and $\qIRT{\ell}{h}$
are defined by first projecting the argument onto a broken polynomial space, and then
performing a suitable averaging on patches of simplices. The patch depends on the considered interpolator but, to allow for a unified presentation, we will only consider the largest patch of all interpolators. If $\sfP\in\Th\cup\Sh$ is a polytopal element or a simplex, we therefore define the (4-level) \emph{simplicial patch} around $\sfP$ as
\[
\patch{\sfP}\coloneq \left\{S\in\Sh\,:\,\exists (S_i)_{i=1,\ldots,4}\in\Sh\text{ s.t. }S_1=S\,,\;\overline{S}_i\cap\overline{S}_{i+1}\not=\emptyset\ \forall i=1,2,3\,,\;
\overline{S}_4\cap\overline{\sfP}\not=\emptyset\right\}.
\]
In a similar way as for the polytopal patch, we will identify the simplicial patch $\patch{\sfP}$ with the (interior of the closure of the) domain $\cup_{S\in \patch{\sfP}}S$ it covers. The context will avoid any confusion between the two concepts.

In the following, we assume that the simplicial submesh $\Sh$ is fine enough so that, for all $T\in\Th$, the simplicial patch $\patch{S}$ around $S\in \Sh(T)$ is topologically trivial, and contained in the polytopal patch $\ppatch{T}$ around $T$ (see \eqref{eq:def.ppatch}). This is not a loss of generality as any given simplicial submesh of $\Th$ can be refined to achieve this property, which also ensures that $\patch{T}$ is contained in $\ppatch{T}$.

\begin{remark}[Generic notation and local interpolator]
If $S\in\Sh$ and $z\in\HgenG{\Omega}$, the quantity $(\qIFEgen{\ell}{h}z)|_S$ actually only depends on the values of $z$ on $\patch{S}$, and is therefore well-defined for all $z\in\HgenG{\patch{S}}$; we will denote this local interpolated function by $\qIFEgen{\ell}{S}z$. 
As a consequence, when $z\in\HgenG{\patch{T}}$ we can define $\qIFEgen{\ell}{T}z\in\FEgenG{\ell}{\Sh(T)}$ such that $(\qIFEgen{\ell}{T}z)|_S=\qIFEgen{\ell}{S}z$ for all $S\in \Sh(T)$. The local interpolator $\qIgen{k}{T}z$ then refers to $\Igen{k}{T}(\qIFEgen{\ell}{T}z)$.
\end{remark}

\begin{theorem}[Bounded cochain map]\label{thm:qI.cochain}
The quasi-interpolators \eqref{eq:def.qI} are well-defined and form a bounded cochain map, that is: the following diagram commutes,
\begin{equation}\label{eq:cochain.map}
  \begin{tikzcd}[column sep=2.2em]
    & 0 \arrow{r}{}
    & \HgradG{\Omega}\arrow{r}{\GRAD}\arrow{d}{\qIgrad{k}{h}}
    & \HcurlG{\Omega}\arrow{r}{\CURL}\arrow{d}{\qIcurl{k}{h}}
    & \HdivG{\Omega}\arrow{r}{\DIV}\arrow{d}{\qIdiv{k}{h}}
    &\L^2(\Omega)\arrow{r}{}\arrow{d}{\lproj{k}{\Th}}
    & 0\\
    & 0 \arrow{r}{}
    & \XgradG{k}{\Th}\arrow{r}{\uG{k}{h}}
    & \XcurlG{k}{\Th}\arrow{r}{\uC{k}{h}}
    & \XdivG{k}{\Th}\arrow{r}{\D{k}{h}}
    &\Poly{k}(\Th)\arrow{r}{}
    & 0
  \end{tikzcd}
\end{equation}
and the quasi-interpolators are uniformly bounded in $h$.
More precisely, we have the following local bounds: for all $\gensymbol\in\{\GRAD,\CURL,\DIV\}$ and $T\in\Th$,
\begin{equation}\label{eq:local.qI.bound}
\tnorm{\gensymbol,T}{\qIgen{k}{T}z}\lesssim \norm{L^2(\patch{T})}{z} + h_T \norm{L^2(\patch{T})}{\gensymbol z} \qquad\forall z\in \HgenG{\patch{T}}.
\end{equation}
Consequently, we also have the following global bound:
\begin{equation}\label{eq:global.qI.bound}
\tnorm{\gensymbol,h}{\qIgen{k}{h}z}\lesssim \norm{L^2(\Omega)}{z}+h\norm{L^2(\Omega)}{\gensymbol z} \qquad\forall z\in \HgenG{\Omega}.
\end{equation}
\end{theorem}

\begin{proof}
See Section \ref{sec:qI.well.posed}.
\end{proof}

\subsection{Consistency properties}

We next want to state consistency properties of the quasi-interpolators. For this, we need to introduce a few approximation measures. 
Recalling the definition \eqref{eq:def.ppatch} of polytopal patch $\ppatch{T}$ around an element $T\in\Th$, and the definition \eqref{eq:def.kgen} of $k_\gensymbol$, the local best approximation error, by broken polynomials, of a function and its differential is
\begin{equation}\label{eq:gen.approx.error}
\Appr{k}{\gensymbol,T}(z)\coloneq\left(\sum_{T'\in\ppatch{T}}\left[\norm{L^2(T')}{z-\lproj{k_\gensymbol}{T'}z}^2+ h_T^2\norm{L^2(T')}{(\gensymbol z)- \lproj{k_\gensymbol}{T'}(\gensymbol z)}^2\right]\right)^{1/2}\quad\forall z\in \Hgen{\ppatch{T}}.
\end{equation}
The global version of this approximation measure is defined by
\begin{equation}\label{eq:Appr.global}
\Appr{k}{\gensymbol,h}(z)=\left(\sum_{T\in\Th}\Appr{k}{\gensymbol,T}(z)^2\right)^{1/2}\quad\forall z\in \Hgen{\Omega}.
\end{equation}

The second set of quantities measures the global best approximation errors of functions and their relevant differential operator by broken finite element polynomials on the polytopal mesh:
\begin{equation}\label{eq:Bppr.grad}
\Bppr{k}{\GRAD,h}(q)\coloneq \left[\sum_{T\in\Th}\min_{q_T\in\Poly{k+1}(T)}
 \norm{L^2(T)}{\GRAD(q - q_T)}^2\right]^{1/2}\qquad\forall q\in\Hgrad{\Omega},
\end{equation}
and, for $\gensymbol\in\{\CURL,\DIV\}$,
\begin{equation}\label{eq:Bppr.gen}
\Bppr{k}{\gensymbol,h}(\bvec{v})\coloneq \left[\sum_{T\in\Th}\min_{\bvec{z}_T\in\FEgen{k+1}{T}}\left(\norm{L^2(T)}{\bvec{v} - \bvec{z}_T}^2 + \norm{L^2(T)}{\gensymbol(\bvec{v} - \bvec{z}_T)}^2\right)\right]^{1/2}\qquad \forall\bvec{v}\in\Hgen{\Omega}.
\end{equation}

Finally, to handle non-homogeneous boundary values (when $\Gamma\not=\partial\Omega$) in the adjoint consistency estimates, we need the following quantities, measuring approximation properties of boundary functions:
\begin{alignat}{2}
\label{eq:def.Tppr.1}
\Tppr{k}{h}(\xi)&\coloneq \left[\sum_{F\in\FGc}\norm{L^2(F)}{\xi-\lproj{k}{F}\xi}^2\right]^{1/2}\quad\forall \xi\in L^2(\Gamma^\compl),\\
\label{eq:def.Tppr.2}
\Tppr{k}{\alpha,\DIV,h}(\bvec{\zeta})&\coloneq \left[\sum_{F\in\FGc}\alpha_F^{-1}\norm{L^2(F)}{\bvec{\zeta}-\RTproj{k+1}{F}\bvec{\zeta}}^2\right]^{1/2}\quad\forall \bvec{\zeta}\in \bvec{L}^2(\Gamma^\compl),
\end{alignat}
where $\FGc\coloneq\{F\in\Fhb\,:\,F\subset\Gamma^\compl\}$, $\alpha=(\alpha_F)_{F\in\FGc}$ is a family of strictly positive numbers, and $\bvec{L}^2(\Gamma^\compl)$ denotes the space of square-integrable vector fields that are tangent to $\Gamma^\compl$.

\begin{theorem}[Primal consistency]\label{thm:qI.primal.consistency}
Let $\gensymbol\in\{\GRAD,\CURL,\DIV\}$ and $T\in\Th$. For all $z\in\HgenG{\ppatch{T}}$, it holds that
\begin{subequations}\label{eq:primal.consistencies}
\begin{alignat}{2}
\label{eq:consistency.Pgen}
\norm{L^2(T)}{z-\Pgen{k}{T}\qIgen{k}{T}z}&\lesssim \Appr{k}{\gensymbol,T}(z),\\
\label{eq:consistency.sTgen}
s_{\gensymbol,T}(\qIgen{k}{T}z,\qIgen{k}{T}z)^{1/2}&\lesssim \Appr{k}{\gensymbol,T}(z),\\
\label{eq:consistency.innergen}
\left|(\qIgen{k}{T}z,\dof{y}{T})_{\mu,\gensymbol,T} - \int_T \mu z\cdot \Pgen{k}{T}\dof{y}{T}\right|&\lesssim \Appr{k}{\gensymbol,T}(z)\norm{\mu,\gensymbol,T}{\dof{y}{T}}\qquad \forall \dof{y}{T}\in\Xgen{k}{T},
\end{alignat}
\end{subequations}
where $\mu$ is any function satisfying Assumption \ref{assum:physical.parameter}.
\end{theorem}

\begin{proof}
See Section \ref{sec:proof.primal.consistency}.
\end{proof}

\begin{remark}[Asymptotic behaviour of $\Appr{k}{\gensymbol,T}$]\label{rem:def.Appr}
From \cite[Theorem 1.45]{Di-Pietro.Droniou:20}, if $z\in\Hgen{\ppatch{T}}$ is such that $z|_{T'}\in H^{k_\gensymbol +1}(T')$ for all $T'\in\ppatch{T}$, then
\[
\Appr{k}{\gensymbol,T}(z)\lesssim h_T^{k_\gensymbol+1}\Bigg(\sum_{T'\in\ppatch{T}}\seminorm{H^{k_\gensymbol+1}(T')}{z}^2\Bigg)^{1/2}.
\]
\end{remark}


\begin{remark}[Approximation properties of differential operators]
Taking $q\in\HgradG{\ppatch{T}}$, applying \eqref{eq:consistency.Pgen} with $\gensymbol=\CURL$ to $z=\GRAD q\in\HcurlG{\ppatch{T}}$, and using the cochain map property of the quasi-interpolators, we obtain an estimate on the primal consistency for the reconstructed gradient $\Pcurl{k}{T}\circ\uG{k}{T}$:
\[
\norm{L^2(T)}{\GRAD q-\Pcurl{k}{T}\uG{k}{T}\qIgrad{k}{T}q}\lesssim \Appr{k}{\CURL,T}(\GRAD q).
\]
A similar estimate can be obtained for $\Pdiv{k}{T}\circ\uC{k}{T}$. The estimate for $\D{k}{h}$ directly follows from the commuting diagram \eqref{eq:cochain.map}, which shows that $\D{k}{h}\circ\qIdiv{k}{h}=\lproj{k}{\Th}\circ\DIV$.
\end{remark}

Non-conforming methods, such as polytopal methods, do not satisfy exact integration-by-parts (IBP) formulas. The analysis of numerical schemes based on these methods therefore requires to evaluate the defects, called adjoint consistency errors, in some discrete forms of IBP. Besides accounting for the novel quasi-interpolators and providing estimates using minimal-regularity assumptions, the following theorem extends the results of \cite[Section 6.2]{Di-Pietro.Droniou:23} to generic mixed boundary conditions (not just fully homogeneous boundary conditions).

\begin{theorem}[Adjoint consistency]\label{thm:qI.adjoint.consistency}
Let $\normal_{\Omega}$ be the unit outer normal to $\Omega$ on $\partial\Omega$, and $\mu$ satisfy Assumption \ref{assum:physical.parameter} (with $n_\gensymbol$ adapted to the particular scalar/vector-valued functions considered below).
The following adjoint consistency estimates hold:
\begin{enumerate}[label=\emph{(\roman*)}]
\item\label{it:adj.cons.grad}\emph{Adjoint consistency for the gradient:} let $\bvec{v}\in \Hdivmu{\Omega}\cap \Hcurl{\Omega}$ be such that $\mu\bvec{v}\cdot\normal_\Omega\in L^2(\Gamma^\compl)$. Then, for all $\dof{q}{h}\in\XgradG{k}{\Th}$, it holds that
\begin{subequations}\label{eq:adjoint.consistency}
\begin{equation}
  \label{eq:adjoint.consistency.grad}
  \begin{aligned}
    \Bigg|(\qIcurl{k}{h}\bvec{v},\uG{k}{h}\dof{q}{h})_{\mu,\CURL,h}+&\int_\Omega \DIV (\mu\bvec{v})\,\Pgrad{k+1}{h}\dof{q}{h}-\int_{\Gamma^\compl}\mu\bvec{v}\cdot\normal_{\Omega}\,\tr{k+1}{\partial\Omega}\dof{q}{h}\Bigg|\\
    \lesssim{}&
    \left(\Appr{k}{\CURL,h}(\bvec{v})+\Bppr{k}{\DIV,h}(\mu\bvec{v})\right)\left(\tnorm{\GRAD,h}{\dof{q}{h}}+\tnorm{\CURL,h}{\uG{k}{h}\dof{q}{h}}\right)\\
    &+\Tppr{k}{h}(\mu\bvec{v}\cdot\normal_\Omega)\Bigg(\sum_{F\in\FGc}\tnorm{\GRAD,F}{\dof{q}{F}}^2\Bigg)^{1/2},
  \end{aligned}
\end{equation}
where $(\tr{k+1}{\partial\Omega}\dof{q}{h})|_F=\tr{k+1}{F}\dof{q}{F}$ for all $F\in\Fhb$.

\item\label{it:adj.cons.curl}\emph{Adjoint consistency for the curl:} let $\bvec{w}\in \Hcurlmu{\Omega}\cap \Hdiv{\Omega}$ be such that $\mu\bvec{w}\times\normal_\Omega\in \bvec{L}^2(\Gamma^\compl)$. Then, for all $\vdof{v}{h}\in\XcurlG{k}{\Th}$ and any strictly positive weights $\alpha=(\alpha_F)_{F\in\FGc}$, it holds that
\begin{equation}
\begin{aligned}
  \label{eq:adjoint.consistency.curl}
    \Bigg|(\qIdiv{k}{h}\bvec{w},{}&\uC{k}{h}\vdof{v}{h})_{\mu,\DIV,h}-\int_\Omega \CURL (\mu\bvec{w})\cdot\Pcurl{k}{h}\vdof{v}{h}
    -\int_{\Gamma^\compl}(\mu\bvec{w}\times\normal_{\Omega})\cdot\trt{k}{\partial\Omega}\vdof{v}{h}\Bigg|\\
    \lesssim{}&
    \left(\Appr{k}{\DIV,h}(\bvec{w})+\Bppr{k}{\CURL,h}(\mu\bvec{w})\right)\left(\tnorm{\CURL,h}{\vdof{v}{h}}+\tnorm{\DIV,h}{\uC{k}{h}\vdof{v}{h}}\right)
    \\
    &+\Tppr{k}{\alpha,\DIV,h}(\mu\bvec{w}\times\normal_{\Omega})
      \Bigg(\sum_{F\in\FGc}\alpha_F\tnorm{\CURL,F}{\vdof{v}{F}}^2\Bigg)^{1/2},
\end{aligned}
\end{equation}
where $(\trt{k}{\partial\Omega}\vdof{v}{h})|_F=\trt{k}{F}\vdof{v}{F}$ for all $F\in\Fhb$.

\item\emph{Adjoint consistency for the divergence}: for all $q\in \Hgradmu{\Omega}$ and all $\vdof{w}{h}\in\Xdiv{k}{\Th}$, it holds that
\begin{equation}
\label{eq:adjoint.consistency.div}
    \left|\int_\Omega \mu q\, \D{k}{h}\vdof{w}{h}+\int_\Omega \GRAD (\mu q) \cdot\Pdiv{k}{h}\vdof{w}{h}
    -\int_{\partial\Omega} \mu q\, w_{h,{\rm n},\partial\Omega}\right|
    \lesssim 
    \Bppr{k}{\GRAD,h}(\mu q)\tnorm{\DIV,h}{\vdof{w}{h}}
\end{equation}
where $(w_{h,{\rm n},\partial\Omega})|_F=\omega_{TF} w_F$ for all $F\in\Fhb$ (with $T\in\Th$ the element containing $F$ in its boundary).
\end{subequations}
\end{enumerate}
\end{theorem}

\begin{proof}
See Section \ref{sec:proof.adjoint.consistency}.
\end{proof}

\begin{remark}[Quasi-interpolators in the adjoint consistency error]\label{rem:BC.qI}
As implicitly indicated by the domain of $\bvec{v}$ in \ref{it:adj.cons.grad} of Theorem \ref{thm:qI.adjoint.consistency}, the quasi-interpolator $\qIcurl{k}{h}$ in \eqref{eq:adjoint.consistency.grad} is the one without boundary conditions, that is, defined on $\Hcurl{\Omega}$. We note however that, if $\bvec{v}\in\HcurlS{\Omega}$ for some $\Sigma\subset\partial\Omega$, we could as well consider the quasi-interpolator corresponding to this boundary condition since, in the proof, the only property required on the quasi-interpolator is the consistency property \eqref{eq:consistency.innergen}. The same consideration holds true for the quasi-interpolator $\qIdiv{k}{h}$ in \ref{it:adj.cons.curl}.
\end{remark}

\begin{remark}[Asymptotic behaviours of the volumetric terms $\Appr{k}{\gensymbol,h}$ and $\Bppr{k}{\gensymbol,h}$]\label{rem:asym.Appr}
By Remark \ref{rem:def.Appr} (together with the fact that, for all $T'\in\Th$, $\mathrm{Card}\{T\in\Th\,:\,T'\in\ppatch{T}\}\lesssim 1$ by mesh regularity), and the approximation properties of local polynomial, Nédélec and Raviart--Thomas functions on polyhedral meshes (see \cite[Theorem 1.45]{Di-Pietro.Droniou:20}, \cite[Lemma 8]{Di-Pietro.Droniou:23} and \cite[Lemma 10]{Di-Pietro.Droniou.ea:25}), we have, for $\gensymbol\in\{\GRAD,\CURL,\DIV\}$:
\begin{alignat*}{2}
\Appr{k}{\gensymbol,h}(z)\lesssim{}& h^{k_\gensymbol+1}\seminorm{H^{k_\gensymbol+1}(\Th)}{z}\qquad&\forall z\in\Hgen{\Omega}\cap H^{k_\gensymbol+1}(\Th),\\
\Bppr{k}{\gensymbol,h}(\phi)\lesssim{}& h^{k+1}\times
  \left\{\begin{array}{ll}
    \seminorm{H^{k+2}(\Th)}{\phi}&\text{ if $\gensymbol=\GRAD$}\\
    \seminorm{H^{k+1}(\Th)}{\phi}+\seminorm{H^{k+2}(\Th)}{\phi}&\text{ if $\gensymbol\in\{\CURL,\DIV\}$}
  \end{array}\right.\qquad&\forall \phi\in\Hgen{\Omega}\cap H^{k+2}(\Th).
\end{alignat*}
For $\gensymbol\in\{\CURL,\DIV\}$ and simplicial meshes, the Peetre--Tartar lemma leads to the slightly improved bound
$\Bppr{k}{\gensymbol,h}(\phi)\lesssim h^{k+1}(\seminorm{H^{k+1}(\Th)}{\phi} + \seminorm{H^{k+1}(\Th)}{\gensymbol \phi})$, see~\cite[Theorem 3.14]{Hiptmair:02}.
The proof however relies on the use of a reference element, a concept that is not available for polytopal meshes.
Extending this improved bound on $\Bppr{k}{\gensymbol,h}(\phi)$ to polytopal elements, by ensuring that the hidden constant only depends on the mesh regularity
factor, does not appear to be completely trivial.
\end{remark}

\begin{remark}[Boundary terms for the adjoint consistency of the gradient]\label{rem:asym.Tppr}
If $\mu\bvec{v}\cdot\normal_\Omega\in H^{k+1}(\FGc)$, the approximation properties of $L^2$-projectors on polynomial spaces \cite[Theorem 1.45]{Di-Pietro.Droniou:20} yield
\[
\Tppr{k}{h}(\mu\bvec{v}\cdot\normal_\Omega)\lesssim h^{k+1}\seminorm{H^{k+1}(\FGc)}{\mu\bvec{v}\cdot\normal_\Omega}.
\]
Moreover, using the trace inequality of Lemma \ref{lem:trace.Xgrad} in Appendix \ref{appen:trace}, the estimate \eqref{eq:adjoint.consistency.grad} can alternatively be written
\begin{equation*}
  \begin{aligned}
    \Bigg|(\qIcurl{k}{h}\bvec{v},\uG{k}{h}\dof{q}{h})_{\mu,\CURL,h}+&\int_\Omega \DIV (\mu\bvec{v})\,\Pgrad{k+1}{h}\dof{q}{h}-\int_{\Gamma^\compl}\mu\bvec{v}\cdot\normal_{\Omega}\,\tr{k+1}{\partial\Omega}\dof{q}{h}\Bigg|\\
    \lesssim{}&
    \left(\Appr{k}{\CURL,h}(\bvec{v})+\Bppr{k}{\DIV,h}(\mu\bvec{v})+\Tppr{k}{h}(\mu\bvec{v}\cdot\normal_\Omega)\right)\left(\tnorm{\GRAD,h}{\dof{q}{h}}+\tnorm{\CURL,h}{\uG{k}{h}\dof{q}{h}}\right).
  \end{aligned}
\end{equation*}
\end{remark}

\begin{remark}[Boundary terms for the adjoint consistency of the curl]\label{rem:asym.Tppr.div}
The analysis of the boundary contributions in \eqref{eq:adjoint.consistency.curl} is less straightforward than the one for \eqref{eq:adjoint.consistency.grad} in Remark \ref{rem:asym.Tppr}; the fundamental reason for this is that there is no $L^2$-trace inequality in $\Hcurl{\Omega}$ (and thus no uniform-in-$h$ such inequality in $\Xcurl{k}{\Th}$, contrary to $\Xgrad{k}{\Th}$).

Assume that $\mu\bvec{w}\times\normal_\Omega\in \bvec{H}^{k+1}(\FGc)$ (the bold font indicating, as before, that the vector field is tangent to its domain).
We first note that, since $\vPoly{k}(F)\subset\RT{k+1}{F}$ for all $F\in\Fhb$, the approximation properties of $\RTproj{k+1}{F}$ are at least as good as those of $\lproj{k}{F}$ (see \cite[Theorem 1.45]{Di-Pietro.Droniou:20}), and thus
\[
\Tppr{k}{\alpha,\DIV,h}(\mu\bvec{w}\times\normal_\Omega)\lesssim \left(\sum_{F\in\FGc}\alpha_F^{-1}h_F^{2(k+1)}\seminorm{H^{k+1}(F)}{\mu\bvec{w}\times\normal_\Omega}^2\right)^{1/2}.
\]
Two main choices can then be made for the weights $\alpha$.

\begin{itemize}
\item If $\alpha=(h_F)_{F\in \FGc}$, then
\[
\Tppr{k}{\alpha,\DIV,h}(\mu\bvec{w}\times\normal_\Omega)\lesssim h^{k+\frac12}\seminorm{H^{k+1}(\FGc)}{\mu\bvec{w}\times\normal_\Omega}.
\]
The convergence appears to be sub-optimal by a factor $1/2$ but this choice of $\alpha$ has the following benefit: if $F\in\FGc$ and $T\in\Th$
is the cell containing $F$ in its boundary, then $\alpha_F\tnorm{\CURL,F}{\vdof{v}{F}}^2=h_F\tnorm{\CURL,F}{\vdof{v}{F}}^2\le\tnorm{\CURL,T}{\vdof{v}{T}}^2$ by definition \eqref{eq:def.component.norm} of the components norm. The estimate \eqref{eq:adjoint.consistency.curl} then leads to
\begin{equation*}
\begin{aligned}
    \Bigg|({}&\qIdiv{k}{h}\bvec{w},\uC{k}{h}\vdof{v}{h})_{\mu,\DIV,h}-\int_\Omega \CURL (\mu\bvec{w})\cdot\Pcurl{k}{h}\vdof{v}{h}
    -\int_{\Gamma^\compl}(\mu\bvec{w}\times\normal_{\Omega})\cdot\trt{k}{\partial\Omega}\vdof{v}{h}\Bigg|\\
    {}&\lesssim
    \left(\Appr{k}{\DIV,h}(\bvec{w})+\Bppr{k}{\CURL,h}(\mu\bvec{w})+\Tppr{k}{\alpha,\DIV,h}(\mu\bvec{w}\times\normal_{\Omega})\right)\left(\tnorm{\CURL,h}{\vdof{v}{h}}+\tnorm{\DIV,h}{\uC{k}{h}\vdof{v}{h}}\right),
\end{aligned}
\end{equation*}
an estimate which does not rely on any boundary norm of the test function $\vdof{v}{h}$.

\item If $\alpha= (1)_{F\in\FGc}$, then we recover an optimal rate of convergence of the trace term
\[
\Tppr{k}{\alpha,\DIV,h}(\mu\bvec{w}\times\normal_{\Omega})\lesssim h^{k+1}\seminorm{H^{k+1}(\FGc)}{\mu\bvec{w}\times\normal_{\Omega}},
\]
but, in \eqref{eq:adjoint.consistency.curl}, the $L^2$-like trace norm $\sum_{F\in\FGc}\tnorm{\CURL,F}{\vdof{v}{F}}^2$
of the test function cannot be uniformly bounded in terms of $\tnorm{\CURL,h}{\vdof{v}{h}}+\tnorm{\DIV,h}{\uC{k}{h}\vdof{v}{h}}$ (that would amount
to having a continuous $L^2$-tangential trace inequality in $\Hcurl{\Omega}$).
\end{itemize}
\end{remark}

\section{Analysis of the quasi-interpolators}\label{sec:analysis.qI}

We prove here the results stated in Section \ref{sec:quasi-interpolator}.

\subsection{Proof of the bounded cochain map property (Theorem \ref{thm:qI.cochain})}\label{sec:qI.well.posed}

The well-posedness of the DDR quasi-interpolators, their commutation property and their boundedness
rely on the two diagrams in \eqref{eq:cochain.map.double}.
\begin{equation}\label{eq:cochain.map.double}
  \begin{tikzcd}[column sep=2em]
    &~\arrow[phantom]{d}{\DFE\;\Bigg\{\; }&[-1em] 
     \HgradG{\Omega}\arrow{r}{\GRAD}\arrow{d}{\qILag{\ell}{h}}
    & \HcurlG{\Omega}\arrow{r}{\CURL}\arrow{d}{\qINE{\ell}{h}}
    & \HdivG{\Omega}\arrow{r}{\DIV}\arrow{d}{\qIRT{\ell}{h}}
    &\L^2(\Omega)\arrow{d}{\lproj{\ell}{\Sh}}
    \\
    &~\arrow[phantom]{d}{\DDDR\;\Bigg\{\;\;\; }
    & \LagG{\ell}{\Sh}\arrow{r}{\GRAD}\arrow{d}{\Igrad{k}{h}}
    & \NEG{\ell}{\Sh}\arrow{r}{\CURL}\arrow{d}{\Icurl{k}{h}}
    & \RTG{\ell}{\Sh}\arrow{r}{\DIV}\arrow{d}{\Idiv{k}{h}}
    &\Poly{\ell}(\Sh)\arrow{d}{\lproj{k}{\Th}}
    \\
    &~
    & \XgradG{k}{\Th}\arrow{r}{\uG{k}{h}}
    & \XcurlG{k}{\Th}\arrow{r}{\uC{k}{h}}
    & \XdivG{k}{\Th}\arrow{r}{\D{k}{h}}
    &\Poly{k}(\Th)
  \end{tikzcd}
\end{equation}
The diagram $\DFE$, corresponding to the first two rows in \eqref{eq:cochain.map.double}, is the diagram of the finite element quasi-interpolators, and is therefore known to be well-defined and commutative
\cite{Chaumont-Frelet.Vohralik:24,Chaumont-Frelet.Vohralik:25,ern_gudi_smears_vohralik_2022a}.
Since the DDR quasi-interpolators \eqref{eq:def.qI} correspond to the composition of the vertical maps in \eqref{eq:cochain.map.double}, their well-posedness and commutation property are ensured if we justify that the bottom two rows, denoted by $\DDDR$, form a diagram of cochain maps.

As can be seen from their definition \eqref{eq:interpolators}, the standard DDR interpolators $\Igen{k}{h}$ are well-defined provided that their arguments are functions with suitable single-valued traces. Let us focus on the case $\gensymbol=\CURL$; the function $\bvec{v}$ to interpolate must have single-valued integrable tangent traces on all edges and faces, and must be integrable in each element. Let us consider $\bvec{v}\in\NE{\ell}{\Sh}$, which contains the domain of $\Icurl{k}{h}$ in \eqref{eq:cochain.map.double}. Since $\Sh$ is a matching submesh of $\Th$, any edge $E\in\Eh$ is the union of edges $e$ of simplices in $\Sh$; on each $e$, by definition of the Nédélec space, $\bvec{v}\cdot\tangent_e=\pm\bvec{v}\cdot\tangent_E$ (depending on the relative orientations of $E$ and $e$) has a well-defined polynomial single value. As a consequence, $\bvec{v}\cdot\tangent_E$ is piecewise polynomial on $E$, and therefore well-defined in $L^1(E)$.
Repeating the same argument for the faces in $\Fh$ (unions of faces of simplices in $\Sh$) and the elements in $\Th$ (unions of simplices in $\Sh$), we see that all the relevant traces and values of $\bvec{v}$ are well-defined, with enough regularity for $\Icurl{k}{h}\bvec{v}$ to be properly defined. The same argument can be applied to the other interpolators, showing that all maps in $\DDDR$ are well-defined.

On smooth enough functions, the DDR interpolators $\Igen{k}{h}$ commute with the differential operators; see \cite[Lemma 4]{Di-Pietro.Droniou:23} and \cite{Di-Pietro.Droniou.ea:20}. An inspection of the proofs of this commutation shows that the only required properties on the functions $z$ to interpolate are: (i) integrability of the appropriate traces (scalar, tangential or normal) of $z$ on the polytopal mesh entities (edges, faces), and (ii) integration-by-parts formulas for these traces of $z$ on the polytopal mesh entities. Considering the interpolators in $\DDDR$, we just checked (i) above, and (ii) follows immediately from the fact that, for $T\in\Th$, traces on a polytopal mesh entity $\sfP\in\ET$ or $\sfP\in\FT$ (depending on the trace and considered finite element) of a finite element function on $\Sh(T)$ are in the suitable finite element space on the trace on $\sfP$ of $\Sh(T)$, and therefore belong to the corresponding conforming space on $\sfP$.

This proves that $\DDDR$ is indeed a diagram of cochain maps, and thus that \eqref{eq:cochain.map} is also a diagram of cochain maps.

\medskip

To conclude the proof of Theorem \ref{thm:qI.cochain}, we need to establish the local bound \eqref{eq:local.qI.bound}.
Recall first that, for all $S\in\Sh$ and $z\in\HgenG{\patch{S}}$, the following bound on the finite element quasi-interpolator holds true
(see \cite[Eq.~(2.3)]{Chaumont-Frelet.Vohralik:25}, \cite[Eq.~(3.18)]{Chaumont-Frelet.Vohralik:24} and \cite[Eq.~(3.7)]{ern_gudi_smears_vohralik_2022a}):
\begin{equation}\label{eq:qI.FE.bound}
 \norm{L^2(S)}{\qIFEgen{\ell}{S}z}^2\lesssim \sum_{S'\in\patch{S}}\left(\norm{L^2(S')}{z}^2+h_{S'}^2\norm{L^2(S')}{\gensymbol z}^2\right).
\end{equation}
Let $z\in\HgenG{\patch{T}}$. Apply \eqref{eq:qI.FE.bound} to each $S\in\Sh(T)$ and sum over these simplices to get
\[
\norm{L^2(T)}{\qIFEgen{\ell}{T}z}^2\lesssim \sum_{S\in\Sh(T)}\sum_{S'\in\patch{S}}\left(\norm{L^2(S')}{z}^2+h_{S'}^2\norm{L^2(S')}{\gensymbol z}^2\right).
\] 
In the double sum $\sum_{S\in\Sh(T)}\sum_{S'\in\patch{S}}$, only simplices $S'$ in $\patch{T}$ appear, and each one of them appears $\lesssim 1$ times. Hence,
\[
\norm{L^2(T)}{\qIFEgen{\ell}{T}z}^2\lesssim \sum_{S'\in\patch{T}}\left(\norm{L^2(S')}{z}^2+h_T^2\norm{L^2(S')}{\gensymbol z}^2\right)
=\norm{L^2(\patch{T})}{z}^2+h_T^2\norm{L^2(\patch{T})}{\gensymbol z}^2.
\] 
Applying Lemma \ref{lem:bound.I.FE} below to $\xi=\qIFEgen{\ell}{T}z\in\FEgenG{\ell}{\Sh(T)}$ and invoking the inequality above yields \eqref{eq:local.qI.bound}.

\begin{lemma}[$L^2$-bound of DDR interpolators on FE spaces]\label{lem:bound.I.FE}
For all $T\in\Th$ and $\xi\in\FEgen{\ell}{\Sh(T)}$, it holds
\[
\tnorm{\gensymbol,T}{\Igen{k}{T}\xi}\lesssim \norm{L^2(T)}{\xi}.
\]
\end{lemma}

\begin{proof}
To fix the ideas we detail the proof in the case $\gensymbol=\CURL$, the other two cases being identical.
The arguments in the first part of the proof of Theorem \ref{thm:qI.cochain} above show that $\Icurl{k}{T}\xi$ is well-defined
for all $\xi\in\NE{\ell}{\Sh(T)}$, and the definitions \eqref{eq:interpolators} and \eqref{eq:def.component.norm} of this interpolate and of the components norm give
\begin{align*}
\tnorm{\CURL,T}{\Icurl{k}{T}\xi}^2 \lesssim&\norm{L^2(T)}{\RTproj{k}{T}\xi}^2
+h_T\sum_{F\in\FT}\norm{L^2(F)}{\RTproj{k}{F}\xi_{{\rm t},F}}^2+h_T^2\sum_{E\in\ET}\norm{L^2(E)}{\lproj{k}{E}(\xi\cdot\tangent_E)}^2\nonumber\\
\le{}& \norm{L^2(T)}{\xi}^2
+h_T\sum_{F\in\FT}\norm{L^2(F)}{\xi}^2
+h_T^2\sum_{E\in\ET}\norm{L^2(E)}{\xi}^2,
\end{align*}
where the second inequality follows by boundedness of $L^2$-projectors and the fact that the norms of tangential components
are bounded above by the norms of the full vector field. Each $E\in\ET$ (resp.~$F\in\FT$) is the union of edges (resp.~faces) of simplices in $\Sh(T)$, which is a regular mesh with $\lesssim 1$ elements on which $\xi$ is piecewise polynomial. Applying the discrete trace inequality in \cite[Lemma 1.32]{Di-Pietro.Droniou:20} on each of these simplicial edges (resp.~faces) and summing the results yields 
\[
 h_T\sum_{F\in\FT}\norm{L^2(F)}{\xi}^2
+h_T^2\sum_{E\in\ET}\norm{L^2(E)}{\xi}^2\lesssim  \norm{L^2(T)}{\xi}^2,
\]
which concludes the proof.
\end{proof}

\subsection{Consistency properties of the quasi-interpolators}\label{sec:proof.consistency}

This section is dedicated to the proof of Theorems \ref{thm:qI.primal.consistency} and \ref{thm:qI.adjoint.consistency}.
We first note the following approximation property on the finite element quasi-interpolators, which is a consequence of
\cite[Eq.~(2.3)]{Chaumont-Frelet.Vohralik:25}, \cite[Eq.~(3.18)]{Chaumont-Frelet.Vohralik:24} and \cite[Eq.~(3.7)]{ern_gudi_smears_vohralik_2022a} together with the fact that $\ell\ge k+1$ (so that $\Poly{k_\gensymbol}(S)\subseteq\FEgen{\ell}{S}$):
for $\gensymbol\in\{\GRAD,\CURL,\DIV\}$, $S\in\Sh$ and $z\in\HgenG{\patch{S}}$, we have
\[
\norm{L^2(S)}{z-\qIFEgen{\ell}{S}z}^2 
\lesssim
\sum_{S'\in\patch{S}}\left(\norm{L^2(S')}{z-\lproj{k_\gensymbol}{S'}z}^2+h_{S'}^2\norm{L^2(S')}{(\gensymbol z)-\lproj{k_\gensymbol}{S'}(\gensymbol z)}^2\right).
\]
For $z\in\HgenG{\ppatch{T}}$, summing these estimates over $S\in\Sh(T)$ and noticing that all the simplices $S'$ that appear in the right-hand side belong to $\patch{T}$, and that each one appears $\lesssim 1$ times, we first infer
\[
\norm{L^2(T)}{z-\qIFEgen{\ell}{T}z}^2 
\lesssim
\sum_{S'\in\patch{T}}\left(\norm{L^2(S')}{z-\lproj{k_\gensymbol}{S'}z}^2+h_{S'}^2\norm{L^2(S')}{(\gensymbol z)-\lproj{k_\gensymbol}{S'}(\gensymbol z)}^2\right).
\]
The sum $\sum_{S'\in\patch{T}}$ can be bounded above by $\sum_{T'\in\ppatch{T}}\sum_{S'\in\Sh(T')}$. Moreover, if $S'\in \Sh(T')$ then, for $Z=z$ or $Z=\gensymbol z$ we have $\norm{L^2(S')}{Z-\lproj{k_\gensymbol}{S'}Z}\le \norm{L^2(S')}{Z-\lproj{k_\gensymbol}{T'}Z}$ (since $\lproj{k_\gensymbol}{S'}Z$ is the best $L^2(S')$-approximation of degree $k_\gensymbol$ of $Z$, and $(\lproj{k_\gensymbol}{T'}Z)|_{S'}$ is another such approximation). Combining these remarks together and recalling \eqref{eq:gen.approx.error}, we finally obtain
\begin{equation}\label{eq:qI.FE.approx}
\norm{L^2(T)}{z-\qIFEgen{\ell}{T}z}
\lesssim \Appr{k}{\gensymbol,T}(z).
\end{equation}

\subsubsection{Proof of the primal consistency (Theorem \ref{thm:qI.primal.consistency})}\label{sec:proof.primal.consistency}

Let us start with the approximation properties \eqref{eq:consistency.Pgen} of the potential reconstruction. For $z\in\HgenG{\ppatch{T}}$ we write, on $T$,
\[
z-\Pgen{k}{T}\qIgen{k}{T}z=z-\lproj{k_\gensymbol}{T}z+\lproj{k_\gensymbol}{T}z-\Pgen{k}{T}\Igen{k}{T}\qIFEgen{\ell}{T}z
=z-\lproj{k_\gensymbol}{T}z+\Pgen{k}{T}\Igen{k}{T}(\lproj{k_\gensymbol}{T}z-\qIFEgen{\ell}{T}z)
\]
where, in the last equality, we have used the polynomial consistency of $\Pgen{k}{T}$ (see \cite[Eqs.~(4.2), (4.7), (4.12)]{Di-Pietro.Droniou:23}) to write $\Pgen{k}{T}\Igen{k}{T}\lproj{k_\gensymbol}{T}z=\lproj{k_\gensymbol}{T}z$. We then take the $L^2(T)$-norm and use the triangle inequality to obtain
\begin{align*}
\norm{L^2(T)}{z-\Pgen{k}{T}\qIgen{k}{T}z}&\le
\norm{L^2(T)}{z-\lproj{k_\gensymbol}{T}z}+\norm{L^2(T)}{\Pgen{k}{T}\Igen{k}{T}(\lproj{k_\gensymbol}{T}z-\qIFEgen{\ell}{T}z)}\\
\overset{\eqref{eq:bound.Pgen.sTgen}}&\lesssim
\norm{L^2(T)}{z-\lproj{k_\gensymbol}{T}z}+\tnorm{\gensymbol,T}{\Igen{k}{T}(\lproj{k_\gensymbol}{T}z-\qIFEgen{\ell}{T}z)}.
\end{align*}
We then invoke Lemma \ref{lem:bound.I.FE} with $\xi=\lproj{k_\gensymbol}{T}z-\qIFEgen{\ell}{T}z\in\FEgen{\ell}{\Sh(T)}$ (recall that $\Poly{k_\gensymbol}(S)\subseteq\FEgen{\ell}{S}$ since $\ell\geq k+1$) to infer
\begin{align*}
\norm{L^2(T)}{z-\Pgen{k}{T}\qIgen{k}{T}z}&\lesssim
\norm{L^2(T)}{z-\lproj{k_\gensymbol}{T}z}+\norm{L^2(T)}{\lproj{k_\gensymbol}{T}z-\qIFEgen{\ell}{T}z}\\
&\lesssim \norm{L^2(T)}{z-\lproj{k_\gensymbol}{T}z}+\norm{L^2(T)}{z-\qIFEgen{\ell}{T}z},
\end{align*}
the second inequality following by introducing $\pm z$ in the second term above and by using the triangle inequality.
Recalling \eqref{eq:qI.FE.approx} concludes the proof of \eqref{eq:consistency.Pgen}.

\medskip

To establish the consistency property \eqref{eq:consistency.sTgen} of the stabilisation bilinear form, we write
\begin{align*}
s_{\gensymbol,T}\left(\qIgen{k}{T}z,\qIgen{k}{T}z\right)^{1/2}&=
s_{\gensymbol,T}\left(\Igen{k}{T}\qIFEgen{\ell}{T}z,\Igen{k}{T}\qIFEgen{\ell}{T}z\right)^{1/2}\\
\overset{\eqref{eq:polynomial.consistency.sT}}&=
s_{\gensymbol,T}\left(\Igen{k}{T}(\qIFEgen{\ell}{T}z-\lproj{k_\gensymbol}{T}z),\Igen{k}{T}(\qIFEgen{\ell}{T}z-\lproj{k_\gensymbol}{T}z)\right)^{1/2}\\
\overset{\eqref{eq:bound.Pgen.sTgen}}&\lesssim
\tnorm{\gensymbol,T}{\Igen{k}{T}(\qIFEgen{\ell}{T}z-\lproj{k_\gensymbol}{T}z)}.
\end{align*}
The conclusion then follows as above, by invoking Lemma \ref{lem:bound.I.FE} and the approximation property \eqref{eq:qI.FE.approx}.

\medskip

The consistency \eqref{eq:consistency.innergen} is a direct consequence of the definition \eqref{eq:def.L2.inner} of the local inner product, the consistencies \eqref{eq:consistency.Pgen} and \eqref{eq:consistency.sTgen} of the potential reconstruction and the stabilisation bilinear form, the definition \eqref{eq:def.L2.norm} of the local $L^2$-like norm, and the fact that $\mu$ is bounded.

\subsubsection{Proof of the adjoint consistency (Theorem \ref{thm:qI.adjoint.consistency})}\label{sec:proof.adjoint.consistency}

\emph{(i) Adjoint consistency for the gradient.}

Let $\mathcal{E}_{\GRAD,h}(\bvec{v},\dof{q}{h})$ be the argument of the absolute value in the left-hand side of \eqref{eq:adjoint.consistency.grad}. We first note that
\begin{equation}\label{eq:consistency.grad.0}
\begin{aligned}
\mathcal{E}_{\GRAD,h}(\bvec{v},\dof{q}{h}) ={}& \sum_{T\in\Th}\left(\int_T \mu\bvec{v}\cdot\Pcurl{k}{T}\uG{k}{T}\dof{q}{T}
+\int_T \DIV (\mu\bvec{v})\,\Pgrad{k+1}{T}\dof{q}{T}\right)+\term_0\\
&-\sum_{T\in\Th}\sum_{F\in\FT\cap\FGc}\omega_{TF}\int_F\mu\bvec{v}\cdot\normal_F\,\tr{k+1}{F}\dof{q}{F},
\end{aligned}
\end{equation}
where
\begin{equation}\label{eq:consistency.grad.term0}
|\term_0|=\left|\sum_{T\in\Th}\left((\qIcurl{k}{T}\bvec{v},\uG{k}{T}\dof{q}{T})_{\mu,\CURL,T}-\int_T  \mu\bvec{v}\cdot\Pcurl{k}{T}\uG{k}{T}\dof{q}{T}\right)\right|
\overset{\eqref{eq:consistency.innergen}}\lesssim \Appr{k}{\CURL,h}(\bvec{v})\norm{\mu,\CURL,h}{\uG{k}{h}\dof{q}{h}}.
\end{equation}

We then recall the following relation \cite[Remark 17]{Di-Pietro.Droniou:23} (see also \cite[Eq.~(4.29)]{Di-Pietro.Droniou:23} to express as $\Pcurl{k}{T}\uG{k}{T}$ the full gradient reconstruction $\boldsymbol{\mathsf{G}}_T^k$ appearing in this relation):
\[
 \int_T\bvec{z}_T\cdot\Pcurl{k}{T}\uG{k}{T}\dof{q}{T}+\int_T\DIV\bvec{z}_T\,\Pgrad{k+1}{T}\dof{q}{T} - \sum_{F\in\FT}\omega_{TF}\int_F
\bvec{z}_T\cdot\normal_F\tr{k+1}{F}\dof{q}{F}=0\quad\forall\bvec{z}_T\in\RT{k+1}{T}.
\]
Subtracting the relations above in each element $T$ from \eqref{eq:consistency.grad.0}, we obtain
\begin{equation}
\mathcal{E}_{\GRAD,h}(\bvec{v},\dof{q}{h})=\term_0+\term_1+\term_2,
\label{eq:consistency.grad.1}
\end{equation}
with $\term_1$ gathering the volumetric terms and $\term_2$ the interface and boundary terms:
\begin{align*}
\term_1={}&
\sum_{T\in\Th}\int_T \left(\mu\bvec{v}-\bvec{z}_T\right)\cdot\Pcurl{k}{T}\uG{k}{T}\dof{q}{T}
+\sum_{T\in\Th}\int_T\DIV(\mu\bvec{v}-\bvec{z}_T)\Pgrad{k+1}{T}\dof{q}{T},\\
\term_2={}&\sum_{T\in\Th}\sum_{F\in\FT}\omega_{TF}\int_F  \bvec{z}_T\cdot\normal_F\,\tr{k+1}{F}\dof{q}{F}
-\sum_{T\in\Th}\sum_{F\in\FT\cap\FGc}\omega_{TF}\int_F \mu\bvec{v}\cdot\normal_F\,\tr{k+1}{F}\dof{q}{F},
\end{align*}
for any family $(\bvec{z}_T)_{T\in\Th}\in\bigtimes_{T\in\Th}\RT{k+1}{T}$.
Using Cauchy--Schwarz inequalities together with \eqref{eq:bound.Pgen.sTgen} to write $\norm{L^2(T)}{\Pcurl{k}{T}\uG{k}{T}\dof{q}{T}} \lesssim \tnorm{\CURL,T}{\uG{k}{T}\dof{q}{T}}$ and $\norm{L^2(T)}{\Pgrad{k+1}{T}\dof{q}{T}} \lesssim \tnorm{\GRAD,T}{\dof{q}{T}}$, we have
\begin{equation}\label{eq:consistency.grad.term}
|\term_1|
\lesssim{}
\left[\sum_{T\in\Th}\left(\norm{L^2(T)}{\mu\bvec{v}-\bvec{z}_T}^2+\norm{L^2(T)}{\DIV(\mu\bvec{v}-\bvec{z}_T)}^2\right)\right]^{1/2}\left(\tnorm{\GRAD,h}{\dof{q}{h}}+\tnorm{\CURL,h}{\uG{k}{h}\dof{q}{h}}\right).
\end{equation}
We now turn to $\term_2$. By \cite[Theorem 3 and Remark 5]{Di-Pietro.Droniou.ea:25} (take dimension $n=3$ and form degree $k=0$ in this reference), there exists a lifting
$\Rgrad:\Xgrad{k}{\Th}\to \Hgrad{\Omega}$ such that (\footnote{The reference \cite{Di-Pietro.Droniou.ea:25} actually considers a trace reconstruction $\tr{k}{F}\dof{q}{F}$ of degree $k$ (not $k+1$ as in \cite{Di-Pietro.Droniou:23} and here) and establishes that $\lproj{k}{F}\Rgrad\dof{q}{h}=\tr{k}{F}\dof{q}{F}$. Since it can be easily checked from the definitions in these references that $\lproj{k}{F}\tr{k+1}{F}\dof{q}{F}=\tr{k}{F}\dof{q}{F}$, this proves that \eqref{eq:lift.Rgrad.proj} holds true.}), 
\begin{align}
\label{eq:lift.Rgrad.proj}
&\lproj{k}{F}\Rgrad\dof{q}{h}=\lproj{k}{F}\tr{k+1}{F}\dof{q}{F}\qquad\forall F\in\Fh,\\
\label{eq:lift.Rgrad.trace.norm}
&\norm{L^2(F)}{\Rgrad\dof{q}{h}}\lesssim\tnorm{\GRAD,F}{\dof{q}{F}}\qquad\forall F\in\Fh,\\
\label{eq:lift.Rgrad.norm}
&\norm{L^2(\Omega)}{\Rgrad\dof{q}{h}}\lesssim\tnorm{\GRAD,h}{\dof{q}{h}}
\quad\text{ and }\quad
\norm{L^2(\Omega)}{\GRAD\Rgrad\dof{q}{h}}\lesssim\tnorm{\CURL,h}{\uG{k}{h}\dof{q}{h}}.
\end{align}
Since $\dof{q}{h}\in\XgradG{k}{\Th}$, the relation \eqref{eq:lift.Rgrad.trace.norm} shows that $\Rgrad\dof{q}{h}\in\HgradG{\Omega}$.

By \cite[Proposition 8]{Di-Pietro.Droniou:23}, we have $\bvec{z}_T\cdot\normal_F\in \Poly{k}(F)$, so the definition of $\lproj{k}{F}$ and the relation \eqref{eq:lift.Rgrad.proj} allow us to write
\begin{align}
\term_2={}&\sum_{T\in\Th}\sum_{F\in\FT}\omega_{TF}\int_F  \bvec{z}_T\cdot\normal_F\,\Rgrad\dof{q}{h}
-\sum_{T\in\Th}\sum_{F\in\FT\cap\FGc}\omega_{TF}\int_F\mu\bvec{v}\cdot\normal_F\, \tr{k+1}{F}\dof{q}{F}\nonumber\\
={}&\sum_{T\in\Th}\langle  \bvec{z}_T\cdot\normal_T,\Rgrad\dof{q}{h}\rangle_{H^{-1/2}(\partial T),H^{1/2}(\partial T)}
-\int_{\Gamma^\compl} \mu\bvec{v}\cdot\normal_\Omega\,\Rgrad\dof{q}{h}+\term_3,
\label{eq:consistency.grad.11}
\end{align}
where $\normal_T$, such that $(\normal_T)|_F\coloneq \omega_{TF}\normal_F$ for all $F\in\FT$, is the outer unit normal to $\partial T$, and
\begin{align}
|\term_3| &= \left|\sum_{T\in\Th}\sum_{F\in\FT\cap\FGc}\omega_{TF}\int_F \mu\bvec{v}\cdot\normal_F\,\big(\Rgrad\dof{q}{h}-\tr{k+1}{F}\dof{q}{F}\big)\right|\nonumber\\
\overset{\eqref{eq:lift.Rgrad.proj}}&= \left|\sum_{T\in\Th}\sum_{F\in\FT\cap\FGc}\omega_{TF}\int_F [(\mu\bvec{v}\cdot\normal_F)-\lproj{k}{F}(\mu\bvec{v}\cdot\normal_F)]\,[\Rgrad\dof{q}{h}-\tr{k+1}{F}\dof{q}{F}]\right|\nonumber\\
\overset{\eqref{eq:def.Tppr.1}}&
\lesssim \Tppr{k}{h}(\mu\bvec{v}\cdot\normal_\Omega) \Bigg(\sum_{F\in\FGc}\tnorm{\GRAD,F}{\dof{q}{F}}^2\Bigg)^{1/2},
\label{eq:consistency.gradient.T3}
\end{align}
where, in the conclusion, we have additionally used a Cauchy--Schwarz inequality, the relation \eqref{eq:lift.Rgrad.trace.norm}, as well as $\norm{L^2(F)}{\tr{k+1}{F}\dof{q}{F}}\lesssim \tnorm{\GRAD,F}{\dof{q}{F}}$ (see \cite[Proposition 6]{Di-Pietro.Droniou:23}).
Since $\Rgrad\dof{q}{h}\in \Hgrad{\Omega}$ and $\mu\bvec{v}\in \Hdiv{\Omega}$, integrations-by-parts in each $T\in\Th$ give
\begin{align*}
\sum_{T\in\Th}\langle\mu\bvec{v}\cdot\normal_T,\Rgrad\dof{q}{h}\rangle_{H^{-1/2}(\partial T),H^{1/2}(\partial T)}={}&
\sum_{T\in\Th}\left(\int_T \big(\mu\bvec{v}\cdot\GRAD\Rgrad\dof{q}{h}+ \DIV(\mu\bvec{v})\Rgrad\dof{q}{h}\big)\right)\\
={}&\int_\Omega \big(\mu\bvec{v}\cdot\GRAD\Rgrad\dof{q}{h}+ \DIV(\mu\bvec{v})\Rgrad\dof{q}{h}\big)\\
={}&\int_{\Gamma^\compl} \mu\bvec{v}\cdot\normal_\Omega\,\Rgrad\dof{q}{h},
\end{align*}
the conclusion following from an integration-by-parts in $\Omega$ together with $\Rgrad\dof{q}{h}\in\HgradG{\Omega}$ and the assumption that $\mu\bvec{v}\cdot\normal_\Omega\in L^2(\Gamma^\compl)$. Using this relation to substitute the boundary term in  \eqref{eq:consistency.grad.11} leads to
\begin{equation*}
\term_2
=\sum_{T\in\Th}\langle (\bvec{z}_T-\mu\bvec{v})\cdot\normal_T, \Rgrad\dof{q}{h}\rangle_{H^{-1/2}(\partial T),H^{1/2}(\partial T)}+\term_3.
\end{equation*}
We then apply element-wise integration-by-parts to get
\begin{equation*} 
\term_2
=\sum_{T\in\Th}\left(\int_T\DIV(\bvec{z}_T-\mu\bvec{v})\,\Rgrad\dof{q}{h}+\int_T(\bvec{z}_T-\mu\bvec{v})\cdot \GRAD\Rgrad\dof{q}{h} \right)
+\term_3.
\end{equation*}
Cauchy--Schwarz inequalities then yield
\begin{align*}
|\term_2|\lesssim |\term_3|+{}&\left[\sum_{T\in\Th}\left(\norm{L^2(T)}{\bvec{z}_T-\mu\bvec{v}}^2+\norm{L^2(T)}{\DIV(\bvec{z}_T-\mu\bvec{v})}^2\right)\right]^{1/2}\\
&\times\left(\norm{L^2(\Omega)}{\Rgrad\dof{q}{h}}+\norm{L^2(\Omega)}{\GRAD\Rgrad\dof{q}{h}}\right).
\end{align*}
Invoking the bounds \eqref{eq:lift.Rgrad.norm} and \eqref{eq:consistency.gradient.T3}, we obtain
\begin{align*}
|\term_2|
\lesssim{}&\Tppr{k}{h}(\mu\bvec{v}\cdot\normal_\Omega)\left(\sum_{F\in\FGc}\tnorm{\GRAD,F}{\dof{q}{F}}^2\right)^{1/2}\\
&+\left[\sum_{T\in\Th}\left(\norm{L^2(T)}{\bvec{z}_T-\mu\bvec{v}}^2+\norm{L^2(T)}{\DIV(\bvec{z}_T-\mu\bvec{v})}^2\right)\right]^{1/2}
\left(\tnorm{\GRAD,h}{\dof{q}{h}}+\tnorm{\CURL,h}{\uG{k}{h}\dof{q}{h}}\right).
\end{align*}
The proof of \eqref{eq:adjoint.consistency.grad} is concluded by plugging this estimate together with \eqref{eq:consistency.grad.term0} (along with~\eqref{eq:coercivity.sT}) and \eqref{eq:consistency.grad.term} into \eqref{eq:consistency.grad.1} and by selecting $(\bvec{z}_T)_{T\in\Th}$ that realise the minima in $\Bppr{k}{\DIV,h}(\mu\bvec{v})$ (see \eqref{eq:Bppr.gen}).

\medskip

\noindent\emph{(ii) Adjoint consistency for the curl}.

We reason in a similar way as for the gradient. With $\mathcal E_{\CURL,h}(\bvec{w},\vdof{v}{h})$ denoting the argument of the absolute value in the left-hand side of \eqref{eq:adjoint.consistency.curl}, we use the consistency \eqref{eq:consistency.innergen} of the local $L^2$-like inner products to write
\begin{equation}\label{eq:consistency.curl.0}
  \begin{aligned}
    \mathcal E_{\CURL,h}(\bvec{w},\vdof{v}{h})={}&\sum_{T\in\Th}\int_T \left(\mu\bvec{w}\cdot\Pdiv{k}{T}\uC{k}{T}\vdof{v}{T}-\int_T\CURL(\mu\bvec{w})\cdot\Pcurl{k}{T}\vdof{v}{T}\right)+\term_0\\
    &-\sum_{T\in\Th}\sum_{F\in\FT\cap\FGc}\omega_{TF}\int_F (\mu\bvec{w}\times\normal_F)\cdot\trt{k}{F}\vdof{v}{F}
  \end{aligned}
\end{equation}
with
\begin{equation}\label{eq:consistency.curl.term0}
|\term_0|\lesssim \Appr{k}{\DIV,h}(\bvec{w})\norm{\mu,\DIV,h}{\uC{k}{h}\vdof{v}{h}}.
\end{equation}

We then recall that \cite[Remark 18]{Di-Pietro.Droniou:23}
\[
\int_T \bvec{z}_T\cdot\Pdiv{k}{T}\uC{k}{T}\vdof{v}{T}-\int_T \CURL\bvec{z}_T\cdot \Pcurl{k}{T}\vdof{v}{T}
-\sum_{F\in\FT}\omega_{TF}\int_F (\bvec{z}_T\times\normal_F)\cdot \trt{k}{F}\vdof{v}{F}=0\quad\forall\bvec{z}_T\in\NE{k+1}{T}.
\]
Subtracting these equations to \eqref{eq:consistency.curl.0} yields
\begin{equation}\label{eq:consistency.curl.1}
	\mathcal E_{\CURL,h}(\bvec{w},\vdof{v}{h})=\term_0+\term_1+\term_2,
\end{equation}
where
\begin{align*}
\term_1={}&\sum_{T\in\Th}\int_T \left((\mu\bvec{w}-\bvec{z}_T)\cdot\Pdiv{k}{T}\uC{k}{T}\vdof{v}{T}-\int_T\CURL(\mu\bvec{w}-\bvec{z}_T)\cdot\Pcurl{k}{T}\vdof{v}{T}\right),\\
\term_2={}&\sum_{T\in\Th}\sum_{F\in\FT}\omega_{TF}\int_F (\bvec{z}_T\times\normal_F)\cdot \trt{k}{F}\vdof{v}{F}
-\sum_{T\in\Th}\sum_{F\in\FT\cap\FGc}\omega_{TF}\int_F (\mu\bvec{w}\times\normal_F)\cdot\trt{k}{F}\vdof{v}{F}.
\end{align*}
The term $\term_1$ is easy to estimate from Cauchy--Schwarz inequalities and~\eqref{eq:bound.Pgen.sTgen}:
\begin{equation}
\label{eq:consistency.curl.term1}
|\term_1|\lesssim \left[\sum_{T\in\Th}\left(\norm{L^2(T)}{\mu\bvec{w}-\bvec{z}_T}^2+\norm{L^2(T)}{\CURL(\mu\bvec{w}-\bvec{z}_T)}^2\right)\right]^{1/2}\left(\tnorm{\CURL,h}{\vdof{v}{h}}+\tnorm{\DIV,h}{\uC{k}{h}\vdof{v}{h}}\right).
\end{equation}
For $\term_2$, we recall the existence of a lifting $\Rcurl:\Xcurl{k}{\Th}\to\Hcurl{\Omega}$ that satisfies the following properties (see \cite[Section 6.5 and Appendix B]{Di-Pietro.Droniou:23} or \cite[Theorem 3 and Remark 5]{Di-Pietro.Droniou.ea:25}, with dimension $n=3$ and form degree $k=1$ in the latter reference):
\begin{align}
\label{eq:Rcurl.trace}
&
  (\Rcurl\vdof{v}{h})_{{\rm t},F}\in \bvec{L}^2(F) \quad\text{ and }\quad
    \RTproj{k+1}{F}(\Rcurl\vdof{v}{h})_{{\rm t},F}=\trt{k}{F}\vdof{v}{F}\qquad\forall F\in\Fh,\\
\label{eq:Rcurl.trace.norm}
&\norm{L^2(F)}{(\Rcurl\vdof{v}{h})_{{\rm t},F}}\lesssim\tnorm{\CURL,F}{\vdof{v}{F}}\qquad\forall F\in\Fh,\\
\label{eq:Rcurl.norm}
&\norm{L^2(\Omega)}{\Rcurl\vdof{v}{h}}\lesssim\tnorm{\CURL,h}{\vdof{v}{h}}
\quad\text{ and }\quad
\norm{L^2(\Omega)}{\CURL\Rcurl\vdof{v}{h}}\lesssim \tnorm{\DIV,h}{\uC{k}{h}\vdof{v}{h}}.
\end{align}
By \eqref{eq:Rcurl.trace.norm}, since $\vdof{v}{h}\in\XcurlG{k}{\Th}$, we have $\Rcurl\vdof{v}{h}\in\HcurlG{\Omega}$.
Noting that $\bvec{z}_T\times\normal_F\in\RT{k+1}{F}$ (see \cite[Proposition 8]{Di-Pietro.Droniou.ea:23}), we use the definition of $\RTproj{k+1}{F}$ and the relation \eqref{eq:Rcurl.trace}, along with the fact that $(\Rcurl\vdof{v}{h})_{{\rm t},\partial\Omega}\in \bvec{L}^2(\partial\Omega)$ (again by \eqref{eq:Rcurl.trace}), to write
\begin{align}
\term_2 ={}& \sum_{T\in\Th}\sum_{F\in\FT}\omega_{TF}\int_F (\bvec{z}_T\times\normal_F)\cdot (\Rcurl\vdof{v}{h})_{{\rm t},F}
-\sum_{T\in\Th}\sum_{F\in\FT\cap\FGc}\omega_{TF}\int_F (\mu\bvec{w}\times\normal_F)\cdot\trt{k}{F}\vdof{v}{F}\nonumber\\
={}&\sum_{T\in\Th}\langle \bvec{z}_T\times\normal_T, (\Rcurl\vdof{v}{h})_{{\rm t},\partial T}\rangle_{H^{-1/2}_\parallel(\partial T),H^{1/2}_\parallel(\partial T)} - \int_{\Gamma^\compl} (\mu\bvec{w}\times\normal_\Omega) \cdot (\Rcurl\vdof{v}{h})_{{\rm t},\partial\Omega}+\term_3,
\label{eq:consistency.curl.term2}
\end{align}
where $\langle\cdot,\cdot\rangle_{H^{-1/2}_\parallel(\partial T),H^{1/2}_\parallel(\partial T)}$ is the duality product between $H^{1/2}_\parallel(\partial T)$ and its dual space (see \cite[Section 3.1]{Assous.Ciarlet.ea:18}), and
\begin{align}
|\term_3| &= \left|\sum_{T\in\Th}\sum_{F\in\FT\cap\FGc}\omega_{TF}\int_F (\mu\bvec{w}\times\normal_F)\cdot[(\Rcurl\vdof{v}{h})_{{\rm t},F}-\trt{k}{F}\vdof{v}{F}]\right|\nonumber\\
\overset{\eqref{eq:Rcurl.trace}}&= \left|\sum_{T\in\Th}\sum_{F\in\FT\cap\FGc}\omega_{TF}\int_F [(\mu\bvec{w}\times\normal_F)-\RTproj{k+1}{F}(\mu\bvec{w}\times\normal_F)]\cdot[(\Rcurl\vdof{v}{h})_{{\rm t},F}-\trt{k}{F}\vdof{v}{F}]\right|\nonumber\\
\overset{\eqref{eq:def.Tppr.2}}&\lesssim 
\Tppr{k}{\alpha,\DIV,h}(\mu\bvec{w}\times\normal_\Omega)\Bigg(\sum_{F\in\FGc}\alpha_F\tnorm{\CURL,F}{\vdof{v}{F}}^2\Bigg)^{1/2},
\label{eq:consistency.curl.term3}
\end{align}
where we have used, in the conclusion, weighted Cauchy--Schwarz inequalities, the bound \eqref{eq:Rcurl.trace.norm}, and $\norm{L^2(F)}{\trt{k}{F}\vdof{v}{F}}\lesssim \tnorm{\CURL,F}{\vdof{v}{F}}$ (see \cite[Proposition 6]{Di-Pietro.Droniou:23}).

Since $\mu\bvec{w}$ and $\Rcurl\vdof{v}{h}$ belong to $\Hcurl{\Omega}$, local integration-by-parts yields
\begin{align*}
\sum_{T\in\Th}\langle \mu\bvec{w}\times\normal_T, (\Rcurl\vdof{v}{h})_{{\rm t},\partial T}\rangle_{H^{-1/2}_\parallel(\partial T),H^{1/2}_\parallel(\partial T)}
={}&
\sum_{T\in\Th}\int_T \left(\mu\bvec{w}\cdot\CURL \Rcurl\vdof{v}{h} -\CURL(\mu\bvec{w})\cdot \Rcurl\vdof{v}{h}\right)\\
={}&
\int_\Omega \left(\mu\bvec{w}\cdot\CURL \Rcurl\vdof{v}{h} -\CURL(\mu\bvec{w})\cdot \Rcurl\vdof{v}{h}\right)\\
={}&
\int_{\Gamma^\compl}(\mu\bvec{w}\times\normal_\Omega)\cdot (\Rcurl\vdof{v}{h})_{{\rm t},\partial \Omega},
\end{align*}
where the conclusion follows from an integration-by-parts over $\Omega$, together with $\Rcurl\vdof{v}{h}\in\HcurlG{\Omega}$ and the $\bvec{L}^2(\Gamma^\compl)$-integrability of $\mu\bvec{w}\times\normal_\Omega$ and $(\Rcurl\vdof{v}{h})_{{\rm t},\partial \Omega}$. Using this relation to substitute the boundary term in \eqref{eq:consistency.curl.term2} and applying again local integrations-by-parts gives
\begin{align*}
\term_2={}&
\sum_{T\in\Th}\langle (\bvec{z}_T-\mu\bvec{w})\times\normal_T, (\Rcurl\vdof{v}{h})_{{\rm t},\partial T}\rangle_{H^{-1/2}_\parallel(\partial T),H^{1/2}_\parallel(\partial T)} +\term_3\\
={}&\sum_{T\in\Th}\left(\int_T (\bvec{z}_T-\mu\bvec{w})\cdot\CURL \Rcurl\vdof{v}{h} - \int_T \CURL(\bvec{z}_T-\mu\bvec{w})\cdot \Rcurl\vdof{v}{h}\right)+\term_3
\end{align*}
from which we get, by Cauchy--Schwarz inequalities, \eqref{eq:Rcurl.norm} and \eqref{eq:consistency.curl.term3},
\begin{align*}
|\term_2|\lesssim{}&
\left(\sum_{T\in\Th}\left(\norm{L^2(T)}{\mu\bvec{w}-\bvec{z}_T}^2+\norm{L^2(T)}{\CURL(\mu\bvec{w}-\bvec{z}_T)}^2\right)\right)^{1/2}\left(\tnorm{\CURL,h}{\vdof{v}{h}}+\tnorm{\DIV,h}{\uC{k}{h}\vdof{v}{h}}\right)\\
&+\Tppr{k}{\alpha,\DIV,h}(\mu\bvec{w}\times\normal_\Omega)\Bigg(\sum_{F\in\FGc}\alpha_F\tnorm{\CURL,F}{\vdof{v}{F}}^2\Bigg)^{1/2}.
\end{align*}
Plugging this bound together with \eqref{eq:consistency.curl.term0} (along with~\eqref{eq:coercivity.sT}) and \eqref{eq:consistency.curl.term1} into \eqref{eq:consistency.curl.1} and selecting $(\bvec{z}_T)_{T\in\Th}$ that realise the minima in $\Bppr{k}{\CURL,h}(\mu\bvec{w})$ concludes the proof of \eqref{eq:adjoint.consistency.curl}.

\medskip

\emph{(iii) Adjoint consistency for the divergence}.

Let $\mathcal{E}_{\DIV,h}(q,\vdof{w}{h})$ be the argument of the absolute value in the left-hand side of
\eqref{eq:adjoint.consistency.div}. The divergence potential satisfies \cite[Eq.~(4.11)]{Di-Pietro.Droniou:23}
\[
\int_Tq_T\,\D{k}{T}\vdof{w}{T}+\int_T \GRAD q_T\cdot\Pdiv{k}{T}\vdof{w}{T}-\sum_{F\in\FT}\omega_{TF}\int_F q_Tw_F = 0\qquad\forall q_T\in\Poly{k+1}(T).
\]
Subtracting this relation in each cell to $\mathcal{E}_{\DIV,h}(q,\vdof{w}{h})$, we find
\begin{align}
\mathcal{E}_{\DIV,h}(q,\vdof{w}{h})={}&
\sum_{T\in\Th}\left(\int_T(\mu q-q_T)\D{k}{T}\vdof{w}{T}+\int_T\GRAD(\mu q-q_T)\cdot\Pdiv{k}{T}\vdof{w}{T}\right)\nonumber\\
&+\sum_{T\in\Th}\left(\sum_{F\in\FT}\omega_{TF}\int_F q_T w_F -\sum_{F\in\FT\cap\Fhb}\omega_{TF}\int_F \mu q w_F\right).
\label{eq:consistency.div.1}
\end{align}
We have
\[
\sum_{T\in\Th}\sum_{F\in\FT}\omega_{TF}\int_F \mu q w_F=\sum_{T\in\Th}\sum_{F\in\FT\cap\Fhb}\omega_{TF}\int_F \mu q w_F,
\]
since the assumption $\mu q\in\Hgrad{\Omega}$ implies that $\mu q$ is single-valued on each face $F$, and thus that the contributions of internal faces above compensate each other. Using this relation to substitute the last term in \eqref{eq:consistency.div.1} we infer
\begin{align*}
\mathcal{E}_{\DIV,h}(q,\vdof{w}{h})={}&
\sum_{T\in\Th}\left(\int_T(\mu q-q_T)\D{k}{T}\vdof{w}{T}+\int_T\GRAD(\mu q-q_T)\cdot\Pdiv{k}{T}\vdof{w}{T}\right)\\
&+\sum_{T\in\Th}\left(\sum_{F\in\FT}\omega_{TF}\int_F (q_T-\mu q) w_F \right).
\end{align*}
Writing $(\mu q-q_T)\D{k}{T}\vdof{w}{T}=h_T^{-1}(\mu q-q_T)h_T\D{k}{T}\vdof{w}{T}$ and $(q_T-\mu q) w_F=h_F^{-1/2}(q_T-\mu q) h_F^{1/2}w_F$, we invoke Cauchy--Schwarz inequalities to get
\begin{align*}
|\mathcal{E}_{\DIV,h}&(q,\vdof{w}{h})|\\
    \lesssim{}&
      \left[
        \sum_{T\in\Th}\left(
        h_T^{-2}\norm{L^2(T)}{\mu q - q_T}^2
        + \norm{L^2(T)}{\GRAD(\mu q - q_T)}^2
        + h_T^{-1}\sum_{F\in\FT}\norm{L^2(F)}{q_T - \mu q}^2
        \right)
        \right]^{\frac12}
      \\
      &\qquad\times\left[
        \sum_{T\in\Th}\left(
        h_T^2\norm{L^2(T)}{\D{k}{T}\vdof{w}{T}}^2
        +\norm{L^2(T)}{\Pdiv{k}{T}\vdof{w}{T}}^2
        + h_T\sum_{F\in\FT}\norm{L^2(F)}{w_F}^2
        \right)
        \right]^{\frac12}.
\end{align*}
The second factor can be bounded by $\tnorm{\DIV,h}{\vdof{w}{h}}$ as in \cite[Proof of Theorem 11]{Di-Pietro.Droniou:23}.
A trace inequality gives $\norm{L^2(F)}{q_T - \mu q}^2\lesssim h_T^{-1}\norm{L^2(T)}{q_T - \mu q}^2+h_T\norm{L^2(T)}{\GRAD(q_T-\mu q)}^2$,
which leads to
\begin{equation*}
|\mathcal{E}_{\DIV,h}(q,\vdof{w}{h})|
    \lesssim{}
      \left[
        \sum_{T\in\Th}\left(
        h_T^{-2}\norm{L^2(T)}{\mu q - q_T}^2
        + \norm{L^2(T)}{\GRAD(\mu q - q_T)}^2
        \right)
        \right]^{\frac12}\tnorm{\DIV,h}{\vdof{w}{h}}.
\end{equation*}
Take now $q_T$ realising $\min_{q_T\in\Poly{k+1}(T)}\norm{L^2(T)}{\GRAD(\mu q-q_T)}^2$, as in the definition of $\Bppr{k}{\GRAD,h}(\mu q)$. We can obviously add a constant to $q_T$ to ensure
that $\int_T(\mu q-q_T)=0$. The local Poincaré--Wirtinger inequality \cite[Remark 1.46]{Di-Pietro.Droniou:20} then gives
$\norm{L^2(T)}{\mu q - q_T}^2\lesssim h_T^2\norm{L^2(T)}{\GRAD(\mu q - q_T)}^2$, and the proof of \eqref{eq:adjoint.consistency.div}
is therefore complete.

\section{Proof of the Maxwell compactness result (Theorem \ref{thm:curl.compactness})}\label{sec:proof.curl.compactness}

Let us start with two preliminary lemmas.

\begin{lemma}[Convergence of approximation errors]\label{lem:Appr.convergence}
Let $\gensymbol\in \{\GRAD,\CURL,\DIV\}$ and recall the definitions \eqref{eq:gen.approx.error}--\eqref{eq:Appr.global}, \eqref{eq:Bppr.grad} and \eqref{eq:Bppr.gen} of the approximation errors. Then, for all $z\in\Hgen{\Omega}$,
\begin{equation}\label{eq:conv.Bppr.z}
\Bppr{k}{\gensymbol,h}(z)\to 0\quad\text{ as $h\to 0$},
\end{equation}
and, if $\mathcal Z$ is a bounded subset of $\Hgen{\Omega}$ that is relatively compact in $L^2(\Omega)$,
\begin{equation}\label{eq:conv.Appr.compact}
\sup_{z\in\mathcal Z}\Appr{k}{\gensymbol,h}(z)\to 0\quad\text{ as $h\to 0$.}
\end{equation}
\end{lemma}

\begin{proof}
By \cite[Chapter I, Theorems 2.4 and 2.10]{Girault.Raviart:86} and \cite[Corollary 9.8]{Brezis:11}, functions in $C^\infty(\overline{\Omega})$ are dense in $\Hgen{\Omega}$. For all $\epsilon>0$, we can therefore find $z_\epsilon\in C^\infty(\overline{\Omega})$
such that $\norm{\Hgen{\Omega}}{z-z_\epsilon}\le \epsilon$. A triangle inequality easily gives $\Bppr{k}{\gensymbol,h}(z)\le \epsilon+\Bppr{k}{\gensymbol,h}(z_\epsilon)$ and, by Remark \ref{rem:asym.Appr}, we infer $\limsup_{h\to 0}\Bppr{k}{\gensymbol,h}(z)\le \epsilon$. Letting $\epsilon\to 0$
concludes the proof of \eqref{eq:conv.Bppr.z}.

Let us now turn to \eqref{eq:conv.Appr.compact}. For all $z\in\mathcal Z$, since for all $T'\in\ppatch{T}$ the projection $\Id-\lproj{k_\gensymbol}{T'}$ is $1$-Lipschitz for the $L^2(T')$-norm, we have
\[
\Appr{k}{\gensymbol,T}(z)\le \left(\sum_{T'\in\ppatch{T}}\Big[\norm{L^2(T')}{z-\lproj{k_\gensymbol}{T'}z}^2
+h^2\norm{L^2(T')}{\gensymbol z}^2\Big]\right)^{1/2}.
\]
Squaring, summing over $T\in\Th$, and using the fact that each $T'\in\Th$ appears in $\lesssim 1$ polytopal patches $(\ppatch{T})_{T\in\Th}$, we infer
\[
\Appr{k}{\gensymbol,h}(z)\lesssim \left(\sum_{T\in\Th}\norm{L^2(T)}{z-\lproj{k_\gensymbol}{T}z}^2\right)^{1/2}
+h\norm{L^2(\Omega)}{\gensymbol z}.
\]
$\mathcal Z$ being pre-compact in $L^2(\Omega)$ and $C^\infty(\overline{\Omega})$ being dense in $L^2(\Omega)$, for $\epsilon>0$ we can find a finite set $\{z_i\,:\,i\in I_\epsilon\}\subset C^\infty(\overline{\Omega})$ such that any $z\in\mathcal Z$ is within $L^2(\Omega)$-distance $\epsilon$ of some $z_i$. A triangle inequality and the fact that $\mathcal Z$ is bounded in $\Hgen{\Omega}$ then yields
\[
\sup_{z\in\mathcal Z}\Appr{k}{\gensymbol,h}(z)\lesssim \epsilon + \sup_{i\in I_\epsilon}\left(\sum_{T\in\Th}\norm{L^2(T)}{z_i-\lproj{k_\gensymbol}{T}z_i}^2\right)^{1/2}+h.
\]
Each $z_i$ being smooth, we have $\sum_{T\in\Th}\norm{L^2(T)}{z_i-\lproj{k_\gensymbol}{T}z_i}^2\to 0$ as $h\to 0$. As $I_\epsilon$ is finite, we infer $\limsup_{h\to 0}\sup_{z\in\mathcal Z}\Appr{k}{\gensymbol,h}(z)\lesssim \epsilon$. Letting $\epsilon$ tend to $0$ concludes the proof of \eqref{eq:conv.Appr.compact}. \end{proof}

\begin{lemma}[Lifting adapted to the quasi-interpolator]\label{lem:lift.interp}
For all $h\in\mathcal H$, there exists a lifting
$\Liftcurl:\XcurlG{k}{\Th}\to \HcurlG{\Omega}$ and a quasi-interpolator $\qIcurl{k}{h}:\HcurlG{\Omega}\to\XcurlG{k}{\Th}$ as in Section \ref{sec:prop.lift}
such that, for all $\vdof{v}{h}\in\XcurlG{k}{\Th}$,
\begin{alignat}{2}
\label{eq:lift.bound}
&\norm{L^2(\Omega)}{\Liftcurl\vdof{v}{h}}\lesssim\tnorm{\CURL,h}{\vdof{v}{h}}
\quad\text{ and }\quad
\norm{L^2(\Omega)}{\CURL\Liftcurl\vdof{v}{h}}\lesssim\tnorm{\DIV,h}{\uC{k}{h}\vdof{v}{h}},\\
\label{eq:lift.interp}
&\qIcurl{k}{h}\Liftcurl\vdof{v}{h}=\vdof{v}{h},\\
\label{eq:proj.lift.Pkh}
&\lproj{k}{\Th} \Liftcurl\vdof{v}{h} = \Pcurl{k}{h}\vdof{v}{h}.
\end{alignat}
\end{lemma}

\begin{proof}
We take $\Liftcurl: \XcurlG{k}{\Th}\to  \Hcurl{\Omega}$
as the restriction to $\XcurlG{k}{\Th}$ of the lifting $\Xcurl{k}{\Th}\to  \Hcurl{\Omega}$ given by \cite[Theorem 3]{Di-Pietro.Droniou.ea:25} for the DDR complex of differential forms (with dimension $n=3$ and form degree $k=1$ in this reference). By \cite[Remarks 4 and 6]{Di-Pietro.Droniou.ea:25}, $\Liftcurl$ has $\NEG{k+3}{\Sh}\subset \HcurlG{\Omega}$ as co-domain.
Recalling the vector proxy representation of differential forms (see, e.g., \cite{Arnold:18} or \cite[Appendix A]{Bonaldi.Di-Pietro.ea:25}), and that for any $r\ge 0$ the trimmed polynomial space of degree $r+1$ contains the full polynomial space of degree $r$, we can use \cite[Theorem 3 and Remark 5]{Di-Pietro.Droniou.ea:25} to see that \eqref{eq:lift.bound} and \eqref{eq:proj.lift.Pkh} hold. 

Let us take the quasi-interpolator $\qIcurl{k}{h}$ constructed as in Section \ref{sec:prop.lift} for $\ell=k+3$. For all $\vdof{v}{h}\in\XcurlG{k}{\Th}$, since $\qINE{\ell}{h}$ is a quasi-interpolator on $\NEG{\ell}{\Sh}$ to which $\Liftcurl\vdof{v}{h}$ belongs, we have $\qINE{\ell}{h}\Liftcurl\vdof{v}{h}=\Liftcurl\vdof{v}{h}$. As $\Icurl{k}{h}\Liftcurl\vdof{v}{h}=\vdof{v}{h}$ by \cite[Remark 5]{Di-Pietro.Droniou.ea:25}, we obtain \eqref{eq:lift.interp} by writing
\[
\qIcurl{k}{h}\Liftcurl\vdof{v}{h}\overset{\eqref{eq:def.qIcurl}}=\Icurl{k}{h}\qINE{\ell}{h}\Liftcurl\vdof{v}{h}=\Icurl{k}{h}\Liftcurl\vdof{v}{h}=\vdof{v}{h}. \qedhere
\]
\end{proof}

We are now ready to prove the Maxwell compactness result for the DDR complex.

\begin{proof}[Proof of Theorem \ref{thm:curl.compactness}]~

\emph{Step 1: weak convergences.}

Let $\Liftcurl$ be the lifting provided by Lemma \ref{lem:lift.interp}. By \eqref{eq:compact.bound} and \eqref{eq:lift.bound}, $(\Liftcurl\vdof{v}{h})_{h\in\mathcal H}$ is bounded in $\HcurlG{\Omega}$ and there exists therefore $\bvec{v}\in\HcurlG{\Omega}$ such that, up to a subsequence as $h\to 0$, $\Liftcurl\vdof{v}{h}\to \bvec{v}$ weakly in $\HcurlG{\Omega}$.

Recalling \eqref{eq:proj.lift.Pkh} and that $\lproj{k}{\Th}$ is self-adjoint for the $L^2(\Omega)^3$-inner product, we note that, for all
$\bvec{\phi}\in L^2(\Omega)^3$,
\[
\int_\Omega \Pcurl{k}{h}\vdof{v}{h}\cdot\bvec{\phi}=\int_\Omega \lproj{k}{\Th}\Liftcurl\vdof{v}{h}\cdot\bvec{\phi}
=\int_\Omega \Liftcurl\vdof{v}{h}\cdot\lproj{k}{\Th}\bvec{\phi}\to \int_\Omega \bvec{v}\cdot\bvec{\phi}\quad\text{as $h\to 0$,}
\]
the conclusion following from the weak $L^2$-convergence $\Liftcurl\vdof{v}{h}\to \bvec{v}$ and the strong $L^2$-convergence $\lproj{k}{\Th}\bvec{\phi}\to \bvec{\phi}$. This proves that $\Pcurl{k}{h}\vdof{v}{h}\to\bvec{v}$ weakly in $L^2(\Omega)^3$.

\medskip

\emph{Step 2: proof that $\bvec{v}\in\HdivGcmu{\Omega}$ and that $\DIV(\mu\bvec{v})\equiv 0$}.

Let $q\in\HgradG{\Omega}$ and apply \eqref{eq:compact.orth} to $\dof{q}{h}=\qIgrad{k}{h}q$ to get
\[
0=(\vdof{v}{h},\uG{k}{h}\qIgrad{k}{h}q)_{\mu,\CURL,h}=
(\vdof{v}{h},\qIcurl{k}{h}(\GRAD q))_{\mu,\CURL,h},
\]
where we have used the cochain property of $\qIgen{k}{h}$ to infer the second equality.
We then apply the consistency \eqref{eq:consistency.innergen} of the local inner products (with $\gensymbol=\CURL$, $z=\GRAD q$ and $\dof{y}{T}=\vdof{v}{T}$), sum over $T\in\Th$, and use the Cauchy--Schwarz inequality together with~\eqref{eq:coercivity.sT}, to infer
\[
\left|\int_{\Omega}\mu\GRAD q\cdot \Pcurl{k}{h}\vdof{v}{h}\right|
\lesssim \Appr{k}{\CURL,h}(\GRAD q)\tnorm{\CURL,h}{\vdof{v}{h}}\overset{\eqref{eq:compact.bound}}\lesssim \Appr{k}{\CURL,h}(\GRAD q).
\]
We have $\Appr{k}{\CURL,h}(\GRAD q)\to 0$ by \eqref{eq:conv.Appr.compact}. Since $\Pcurl{k}{h}\vdof{v}{h}\to \bvec{v}$ weakly in $L^2(\Omega)^3$, we can therefore pass to the limit and use the symmetry of $\mu$ to deduce that
\[
\int_\Omega \GRAD q\cdot\mu\bvec{v}=0.
\]
Applying this to $q\in C^\infty_c(\Omega)$ yields $\DIV (\mu\bvec{v})\equiv 0$, so that $\bvec{v}\in\Hdivmu{\Omega}$, and considering
then a generic $q\in \HgradG{\Omega}$ gives $\mu\bvec{v}\cdot\normal_\Omega\equiv 0$ on $\Gamma^\compl$.

\medskip

\emph{Step 3: strong convergences}.

We start by writing the Hodge decomposition of $\bvec{v}-\Liftcurl\vdof{v}{h}$:
\begin{equation}\label{eq:hodge}
\bvec{v}-\Liftcurl\vdof{v}{h}=\bvec{w}(h)+\GRAD q(h)
\end{equation}
with $\bvec{w}(h)\in\HdivGc{\Omega}$ such that $\DIV\bvec{w}(h)\equiv 0$ and $q(h)\in \HgradG{\Omega}$ (the purpose of the notations $\bvec{w}(h)$ and $q(h)$ is to recall that these functions depend on $h$, but we avoid using the subscript to distinguish them from elements of the discrete spaces). 
Since the Hodge decomposition is orthogonal, the $L^2$-norms of $\bvec{w}(h)$ and $\GRAD q(h)$ are bounded by
$\norm{L^2(\Omega)}{\bvec{v}-\Liftcurl\vdof{v}{h}}\lesssim 1$, owing to \eqref{eq:lift.bound} and \eqref{eq:compact.bound}.
We can always choose $q(h)$ such that its average on each connected component of $\Omega$ that does not touch $\Gamma$ vanishes.
Applying then Poincar\'e inequalities on the connected components that touch $\Gamma$ and Poincar\'e--Wirtinger inequalities on the other connected components, we obtain
\begin{equation}\label{eq:bound.q(h)}
\norm{L^2(\Omega)}{q(h)}\lesssim\norm{L^2(\Omega)}{\GRAD q(h)}\lesssim 1.
\end{equation}
Since $\bvec{v}$, $\Liftcurl\vdof{v}{h}$ and $\GRAD q(h)$ all belong to $\HcurlG{\Omega}$, the relation \eqref{eq:hodge} shows that $\bvec{w}(h)\in \HcurlG{\Omega}$ and 
$\CURL \bvec{w}(h)=\CURL(\bvec{v}-\Liftcurl\vdof{v}{h})$, using $\CURL\GRAD=\bvec{0}$. Invoking \eqref{eq:lift.bound} and \eqref{eq:compact.bound} once again, together with the compactness results of \cite[Theorem 5]{Jochmann:97} (see also \cite{Weck:74,Weber:80} and \cite[Proposition 3.7]{Amrouche.Bernardi.ea:98} for $\Gamma=\partial\Omega$), we infer that
\begin{equation}\label{eq:wh.bound.compact}
  \begin{aligned}
  &\text{$\{\bvec{w}(h)\,:\,h\in\mathcal H\}$ is bounded in $\HcurlG{\Omega}\cap \HdivGc{\Omega}$}\\
  &\text{and (thus) relatively compact in $L^2(\Omega)^3$.}
  \end{aligned}
\end{equation}

By \eqref{eq:hodge} and \eqref{eq:lift.interp}, and using the cochain property of $\qIgen{k}{h}$, we have 
\[
\qIcurl{k}{h}\bvec{v}-\vdof{v}{h}=\qIcurl{k}{h}\Big(\bvec{w}(h) +\GRAD q(h)\Big)=
\qIcurl{k}{h}\bvec{w}(h) +\uG{k}{h} \qIgrad{k}{h}q(h).
\]
Hence,
\begin{align}
\norm{\mu,\CURL,h}{\qIcurl{k}{h}\bvec{v}-\vdof{v}{h}}^2&=
\big(\qIcurl{k}{h}\bvec{v}-\vdof{v}{h},\qIcurl{k}{h}\bvec{w}(h)+\uG{k}{h}\qIgrad{k}{h}q(h)\big)_{\mu,\CURL,h}\nonumber\\
\overset{\eqref{eq:compact.orth}}&=
\big(\qIcurl{k}{h}\bvec{v}-\vdof{v}{h},\qIcurl{k}{h}\bvec{w}(h)\big)_{\mu,\CURL,h}
+ \big(\qIcurl{k}{h}\bvec{v},\uG{k}{h}\qIgrad{k}{h}q(h)\big)_{\mu,\CURL,h}\nonumber\\
&=\term_1+\term_2.
\label{eq:compact.conv.curl.h}
\end{align}

The term $\term_1$ can be estimated using the primal consistency of the inner product \eqref{eq:consistency.innergen} (along with~\eqref{eq:coercivity.sT}):
\begin{align}
\left|\term_1-\int_\Omega \mu\bvec{w}(h)\cdot\big(\Pcurl{k}{h}\qIcurl{k}{h}\bvec{v}-\Pcurl{k}{h}\vdof{v}{h}\big)\right|
&\lesssim \Appr{k}{\CURL,h}(\bvec{w}(h))\tnorm{\CURL,h}{\qIcurl{k}{h}\bvec{v}-\vdof{v}{h}}\nonumber\\
\overset{ \eqref{eq:global.qI.bound}, \eqref{eq:compact.bound}}&\lesssim \Appr{k}{\CURL,h}(\bvec{w}(h)).
\label{eq:compact.bound.T1}
\end{align}
By \eqref{eq:consistency.Pgen} and \eqref{eq:conv.Appr.compact} we have $\Pcurl{k}{h}\qIcurl{k}{h}\bvec{v}\to \bvec{v}$ in $L^2(\Omega)^3$ as $h\to  0$. Recalling that $\Pcurl{k}{h}\vdof{v}{h}\to \bvec{v}$ weakly in $L^2(\Omega)^3$ and that $\mathcal Z\coloneq \{w(h)\,:\,h\in\mathcal H\}$ is relatively compact in $L^2(\Omega)^3$ (see \eqref{eq:wh.bound.compact}), we deduce that
\[
\int_\Omega \mu \bvec{w}(h)\cdot (\Pcurl{k}{h}\qIcurl{k}{h}\bvec{v}-\Pcurl{k}{h}\vdof{v}{h})\to 0\text{ as $h\to 0$}.
\]
Moreover, by \eqref{eq:conv.Appr.compact} and \eqref{eq:wh.bound.compact}, $\Appr{k}{\CURL,h}(\bvec{w}(h))\le \sup_{\bvec{z}\in\mathcal Z}\Appr{k}{\CURL,h}(z)\to 0$ as $h\to 0$. Plugging these convergences into \eqref{eq:compact.bound.T1} shows that $\term_1\to 0$ as $h\to 0$.

Let us now consider $\term_2$. Since $\dof{q}{h}\coloneq\qIgrad{k}{h}q(h)\in\XgradG{k}{\Th}$ and $\bvec{v}\in\Hdivmu{\Omega}\cap\Hcurl{\Omega}$ is such that $\mu\bvec{v}\cdot\normal_\Omega$ belongs to $L^2(\Gamma^\compl)$ (since it vanishes on $\Gamma^\compl$), we can apply the adjoint consistency estimate \eqref{eq:adjoint.consistency.grad} -- invoking Remark \ref{rem:BC.qI} to justify the use of this bound with the considered quasi-interpolator $\qIcurl{k}{h}$ embedding the zero boundary condition on $\Gamma$. Recalling that $\DIV(\mu\bvec{v})\equiv 0$, that $\mu\bvec{v}\cdot\normal_\Omega\equiv 0$ on $\Gamma^\compl$ (which implies $\Tppr{k}{h}(\mu\bvec{v}\cdot\normal_\Omega)=0$), and leveraging the cochain map property of $\qIgen{k}{h}$, \eqref{eq:adjoint.consistency.grad} gives
\[
|\term_2|\lesssim \left(\Appr{k}{\CURL,h}(\bvec{v})+\Bppr{k}{\DIV,h}(\mu \bvec{v})\right)\Big(\underbrace{\tnorm{\GRAD,h}{\qIgrad{k}{h}q(h)}+\tnorm{\CURL,h}{\qIcurl{k}{h}\GRAD q(h)}}_{\lesssim 1\text{ by \eqref{eq:global.qI.bound},\eqref{eq:bound.q(h)}}}\Big).
\]
Lemma \ref{lem:Appr.convergence} then yields $\term_2\to 0$ as $h\to 0$. Together with the convergence $\term_1\to 0$ and
\eqref{eq:compact.conv.curl.h}, this establishes that
\[
\norm{\mu,\CURL,h}{\qIcurl{k}{h}\bvec{v}-\vdof{v}{h}}\to 0\text{ as $h\to 0$}.
\]
Using the definition \eqref{eq:def.L2.norm} of the $\mu$-weighted $L^2$-like norm on $\Xcurl{k}{\Th}$ and the coercivity of $\mu$, we deduce that, as $h\to 0$,
\begin{alignat}{2}
\label{eq:conv.PIv}
&\norm{L^2(\Omega)}{\Pcurl{k}{h}\qIcurl{k}{h}\bvec{v}-\Pcurl{k}{h}\vdof{v}{h}}\to 0,\\
\label{eq:conv.sT}
&s_{\CURL,h}(\qIcurl{k}{h}\bvec{v}-\vdof{v}{h},\qIcurl{k}{h}\bvec{v}-\vdof{v}{h})\to 0.
\end{alignat}
We have established above that $\Pcurl{k}{h}\qIcurl{k}{h}\bvec{v}\to \bvec{v}$ strongly in $L^2(\Omega)^3$, so \eqref{eq:conv.PIv}
concludes the proof that $\Pcurl{k}{h}\vdof{v}{h}\to \bvec{v}$ strongly in the same space. A triangle inequality on the
seminorm $\vdof{z}{h}\to s_{\CURL,h}(\vdof{z}{h},\vdof{z}{h})^{1/2}$ shows that
\[
s_{\CURL,h}(\vdof{v}{h},\vdof{v}{h})^{1/2}\le s_{\CURL,h}(\qIcurl{k}{h}\bvec{v}-\vdof{v}{h},\qIcurl{k}{h}\bvec{v}-\vdof{v}{h})^{1/2}
+s_{\CURL,h}(\qIcurl{k}{h}\bvec{v},\qIcurl{k}{h}\bvec{v})^{1/2}.
\]
The first term in the right-hand side converges to $0$ by \eqref{eq:conv.sT}, and the second also by \eqref{eq:consistency.sTgen}
and \eqref{eq:conv.Appr.compact}. This concludes the proof that $s_{\CURL,h}(\vdof{v}{h},\vdof{v}{h})\to 0$
as $h\to 0$.
\end{proof}

\section*{Acknowledgements}

JD acknowledges the funding of the European Union via the ERC Synergy, NEMESIS, project number 101115663.
Views and opinions expressed are however those of the authors only and do not necessarily reflect those of the European Union or the European Research Council Executive Agency.
Neither the European Union nor the granting authority can be held responsible for them.
The work of SL is supported by the Agence Nationale de la
Recherche (ANR) under the PRCE grant HIPOTHEC (ANR-23-CE46-0013).
TCF and SL also acknowledge support from the LabEx CEMPI (ANR-11-LABX-0007).

\appendix

\section{Discrete trace inequalities in $\Xgrad{k}{\Th}$}\label{appen:trace}

\begin{lemma}[Discrete trace inequalities in $\Xgrad{k}{\Th}$]\label{lem:trace.Xgrad}
If $\Mh$ is a member of a regular polytopal mesh sequence then, for all $\dof{q}{h}\in\Xgrad{k}{\Th}$,
\begin{alignat}{2}
\label{eq:trace.trq}
\sum_{F\in\Fhb}\norm{L^2(F)}{\tr{k+1}{F}\dof{q}{F}}^2\lesssim{}& \tnorm{\GRAD,h}{\dof{q}{h}}^2+\tnorm{\CURL,h}{\uG{k}{h}\dof{q}{h}}^2,\\
\label{eq:trace.tnormq}
\sum_{F\in\Fhb}\tnorm{\GRAD,F}{\dof{q}{F}}^2\lesssim{}& \tnorm{\GRAD,h}{\dof{q}{h}}^2+\tnorm{\CURL,h}{\uG{k}{h}\dof{q}{h}}^2.
\end{alignat}
\end{lemma}

\begin{proof}
Let us start with \eqref{eq:trace.trq}. Let $\dof{q}{h}\in\Xgrad{k}{\Th}$ and set $C=\frac{1}{|\Omega|}\int_\Omega\Pgrad{k+1}{h}\dof{q}{h}$.
Then, $\dof{\tilde{q}}{h}\coloneq \dof{q}{h}-\Igrad{k}{h}C\in\Xgrad{k}{\Th}$ satisfies $\int_\Omega \Pgrad{k+1}{h}\dof{\tilde{q}}{h}=0$, since
$\Pgrad{k+1}{h}\Igrad{k}{h}C=C$ by polynomial consistency of the potential reconstruction \cite[Eq.~(4.2)]{Di-Pietro.Droniou:23}.
Hence, we can apply the trace inequality of \cite[Theorem 6.7]{Di-Pietro.Droniou:20} (with $p=2$) to the hybrid vector
$\big((\Pgrad{k+1}{T}\dof{\tilde{q}}{T})_{T\in\Th},(\tr{k+1}{F}\dof{\tilde{q}}{F})_{F\in\Fh}\big)$ to get
\begin{align}
\sum_{F\in\Fhb}\norm{L^2(F)}{\tr{k+1}{F}\dof{\tilde{q}}{F}}^2&\lesssim \sum_{T\in\Th}\left(\norm{L^2(T)}{\GRAD \Pgrad{k+1}{T}\dof{\tilde{q}}{T}}^2
+\sum_{F\in\FT}h_F^{-1}\norm{L^2(F)}{\Pgrad{k+1}{T}\dof{\tilde{q}}{T}-\tr{k+1}{F}\dof{\tilde{q}}{F}}^2\right)\nonumber\\
&\eqcolon \sum_{T\in\Th}\term_T.
\label{eq:trace.1}
\end{align}
The estimate \cite[Eq.~(5.6)]{Di-Pietro.Droniou:23} shows that $\term_T\lesssim \tnorm{\CURL,T}{\uG{k}{T}\dof{\tilde{q}}{T}}^2= \tnorm{\CURL,T}{\uG{k}{T}\dof{q}{T}}^2$ (using the polynomial consistency \cite[Eq.~(3.16)]{Di-Pietro.Droniou:23} of $\uG{k}{T}$).
In turn, the polynomial consistency of $\tr{k+1}{F}$ (see \cite[Eq.~(3.14)]{Di-Pietro.Droniou:23}) yields $\tr{k+1}{F}\dof{\tilde{q}}{F}=\tr{k+1}{F}\dof{q}{F}-C$.
Plugging these into \eqref{eq:trace.1} and using a triangle inequality gives
\[
\sum_{F\in\Fhb}\norm{L^2(F)}{\tr{k+1}{F}\dof{q}{F}}^2\lesssim |\partial\Omega| C^2+\tnorm{\CURL,h}{\uG{k}{h}\dof{q}{h}}^2
\lesssim \norm{L^2(\Omega)}{\Pgrad{k+1}{h}\dof{q}{h}}^2+\tnorm{\CURL,h}{\uG{k}{h}\dof{q}{h}}^2,
\]
where the conclusion is obtained recalling the definition of $C$ and using a Cauchy--Schwarz inequality. The proof of \eqref{eq:trace.trq} is complete
by recalling \eqref{eq:coercivity.sT} to write $ \norm{L^2(\Omega)}{\Pgrad{k+1}{h}\dof{q}{h}}^2 \lesssim\tnorm{\GRAD,h}{\dof{q}{h}}^2$.

To prove \eqref{eq:trace.tnormq}, we introduce $\pm\Igrad{k}{F}\tr{k+1}{F}\dof{q}{F}$ in the norm below and use a triangle inequality to get
\begin{equation}\label{eq:trace.2}
\tnorm{\GRAD,F}{\dof{q}{F}}^2\lesssim \tnorm{\GRAD,F}{\dof{q}{F}-\Igrad{k}{F}\tr{k+1}{F}\dof{q}{F}}^2+
\tnorm{\GRAD,F}{\Igrad{k}{F}\tr{k+1}{F}\dof{q}{F}}^2.
\end{equation}
On the one hand, the definition \eqref{eq:interpolators} of the interpolator and the property $\lproj{k-1}{F}\tr{k+1}{F}\dof{q}{F}=q_F$ (see \cite[Eq.~(3.15)]{Di-Pietro.Droniou:23}) directly give
\begin{align*}
 \tnorm{\GRAD,F}{\dof{q}{F}-\Igrad{k}{F}\tr{k+1}{F}\dof{q}{F}}^2={}&h_F\sum_{E\in\EF}\norm{L^2(E)}{q_E-\lproj{k-1}{E}\tr{k+1}{F}\dof{q}{F}}^2
 +h_F^2\sum_{V\in\VF}|q_V-\tr{k+1}{F}\dof{q}{F}(\bvec{x}_V)|^2\\
 \lesssim{}& h_F\sum_{E\in\EF}\norm{L^2(E)}{\tr{k+1}{E}\dof{q}{E}-\tr{k+1}{F}\dof{q}{F}}^2,
\end{align*}
where the inequality follows from the definition of $\tr{k+1}{E}\dof{q}{E}$ (which satisfies $\lproj{k-1}{E}\tr{k+1}{E}\dof{q}{E}=q_E$ and $\tr{k+1}{E}\dof{q}{E}(\bvec{x}_V)=q_V$ for all $V\in\VE$) and discrete trace inequalities.
Using \cite[Lemma 7]{Di-Pietro.Droniou:23} (in which $q_E$ is actually $\tr{k+1}{E}\dof{q}{E}$), we infer
\begin{equation}\label{eq:trace.3}
 \tnorm{\GRAD,F}{\dof{q}{F}-\Igrad{k}{F}\tr{k+1}{F}\dof{q}{F}}^2\lesssim
 h_F^2\tnorm{\CURL,F}{\uG{k}{F}\dof{q}{F}}^2.
 \end{equation}
On the other hand, in the same way as in the proof of Lemma \ref{lem:bound.I.FE}, the definitions of the interpolator and of the components norm \eqref{eq:def.component.norm} together with discrete trace inequalities yield
\[
\tnorm{\GRAD,F}{\Igrad{k}{F}\tr{k+1}{F}\dof{q}{F}}^2\lesssim \norm{L^2(F)}{\tr{k+1}{F}\dof{q}{F}}^2.
\]
Plugging this together with \eqref{eq:trace.3} into \eqref{eq:trace.2}, summing over $F\in\Fhb$, and invoking \eqref{eq:trace.trq}, gives
\[
\sum_{F\in\Fhb}\tnorm{\GRAD,F}{\dof{q}{F}}^2\lesssim \sum_{F\in\Fhb} h_F^2\tnorm{\CURL,F}{\uG{k}{F}\dof{q}{F}}^2+
\tnorm{\GRAD,h}{\dof{q}{h}}^2+\tnorm{\CURL,h}{\uG{k}{h}\dof{q}{h}}^2.
\]
The proof of \eqref{eq:trace.tnormq} is complete by noticing that, for $F\in\Fhb$, if $T\in\Th$ is the element that contains $F$ in its boundary, we have
$h_F^2\tnorm{\CURL,F}{\uG{k}{F}\dof{q}{F}}^2\le h_F\tnorm{\CURL,T}{\uG{k}{T}\dof{q}{T}}^2\lesssim \tnorm{\CURL,T}{\uG{k}{T}\dof{q}{T}}^2$ by definition of the components norm. \end{proof}

\printbibliography


\end{document}